\documentclass[11pt,letterpaper,oneside]{amsart}
\usepackage[body={6in,9in}]{geometry}
\usepackage{amssymb,amsmath}
\usepackage{hyperref}
\hypersetup{
	pdfauthor   = {N. Bruin, E.V. Flynn, A. Shnidman},
	pdftitle    = {Genus two curves with full sqrt(3)-level structure},
	pdfsubject  = {},
	pdfkeywords = {},
	backref=true, pagebackref=true, hyperindex=true, colorlinks=true,
	breaklinks=true, urlcolor=blue, linkcolor=blue, citecolor=blue,
	bookmarks=true, bookmarksopen=true}
\usepackage[alphabetic,backrefs,lite]{amsrefs}
\usepackage{latexsym}
\usepackage{color}
\usepackage{graphicx}
\usepackage{fancyhdr}
\usepackage[OT2,OT1]{fontenc}

\DeclareFontFamily{U}{wncy}{}
\DeclareFontShape{U}{wncy}{m}{n}{<->wncyr10}{}
\DeclareSymbolFont{mcy}{U}{wncy}{m}{n}
\DeclareMathSymbol{\Sha}{\mathord}{mcy}{"58}

\DeclareMathOperator{\disc}{disc}

\DeclareMathOperator{\PGL}{PGL}

\DeclareMathOperator{\PSp}{PSp}

\DeclareMathOperator{\End}{End}
\DeclareMathOperator{\Spec}{Spec}
\DeclareMathOperator{\Sel}{Sel}
\DeclareMathOperator{\cH}{H}

\newcommand{\sep}{{\operatorname{sep}}}
\newcommand{\minus}{\scalebox{0.70}[1.0]{$-$}}
\newcommand{\Z}{{\mathbb Z}}
\renewcommand{\AA}{\mathbb{A}}
\newcommand{\ZZ}{{\mathbb Z}}
\newcommand{\Q}{{\mathbb Q}}
\newcommand{\QQ}{{\mathbb Q}}
\newcommand{\FF}{{\mathbb F}}

\newcommand{\RR}{{\mathbb R}}
\newcommand{\PP}{{\mathbb P}}
\newcommand{\sB}{\mathcal{B}}
\newcommand{\sC}{\mathcal{C}}
\newcommand{\sX}{\mathcal{X}}
\newcommand{\sW}{\mathcal{W}}
\newcommand{\sZ}{\mathcal{Z}}

\newcommand{\sO}{\mathcal{O}}

\newcommand{\sH}{\mathcal{H}}
\newcommand{\sJ}{\mathcal{J}}

\newcommand{\C}{C}
\newcommand{\Cmin}{\mathcal{C}^\mathrm{min}}
\newcommand{\Cnaive}{\mathcal{C}^\mathrm{naive}}

\newcommand{\F}{\mathcal{F}}

\newcommand{\sA}{\mathcal{A}}
\newcommand{\sR}{\mathcal{R}}

\newcommand{\Jac}{{\mathrm{Jac}}}
\newcommand{\alg}{\mathrm{alg}}
\newcommand{\JacC}{{\hbox{Jac}_{\lower.5pt\hbox{$_\sC$}}}}
\newcommand{\JacF}{{\hbox{Jac}_{\lower.5pt\hbox{$_\F$}}}}

\newcommand{\ord}{\mathrm{ord}}

\newcommand{\Gal}{\mathrm{Gal}}
\newcommand{\Res}{\mathrm{Res}}

\newcommand{\kred}{\overline{k}}
\newcommand{\fp}{\mathfrak{p}}
\newcommand{\fs}{\mathfrak{s}}
\DeclareMathOperator{\Aut}{Aut}

\newtheorem{thm}{Theorem}[section]
\newtheorem{conj}[thm]{Conjecture}

\newtheorem{lemma}[thm]{Lemma}
\newtheorem{prop}[thm]{Proposition}
\newtheorem{cor}[thm]{Corollary}
\theoremstyle{definition}
\newtheorem{example}[thm]{Example}
\newtheorem{rmk}[thm]{Remark}
\newtheorem{notation}[thm]{Notation}
 
\theoremstyle{definition}
\newtheorem{defn}[thm]{Definition}
\theoremstyle{remark}

\title{Genus two curves with full $\sqrt{3}$-level structure and Tate-Shafarevich groups}

\author{N.\ Bruin}
\address{Department of Mathematics, Simon Fraser University,
Burnaby, BC, CANADA, V5A 1S6}
\email{nbruin@sfu.ca}
\author{E.V.\ Flynn}
\address{Mathematical Institute, University of Oxford, 24--29 St.\ Giles,
Oxford OX1 3LB, United Kingdom}
\email{flynn@maths.ox.ac.uk}
\author{A.\ Shnidman}
\address{Einstein Institute of Mathematics, Hebrew University of Jerusalem, Israel} 
\email{ariel.shnidman@huji.mail.ac.il}

\thanks{The first author acknowledges the support of the Natural Sciences and Engineering Research Council of
	Canada (NSERC), funding reference number RGPIN-2018-04191.  The
third author was supported by the Israel Science Foundation (grant No. 2301/20)}
\date{1 February, 2021}

\begin{document}

\begin{abstract}
We give an explicit rational parameterization of the surface $\mathcal{H}_3$ over $\Q$ whose points parameterize genus 2 curves~$\C$
with full $\sqrt{3}$-level structure on their Jacobian $J$.  
We use this model to construct abelian surfaces $A$ with the property that $\Sha(A_d)[3] \neq 0$ for a positive proportion of quadratic twists $A_d$.  
In fact, for $100\%$ of $x \in \mathcal{H}_3(\Q)$, this holds for the surface $A = \Jac(C_x)/\langle P \rangle$, where $P$ is the marked point of order $3$.  Our methods also give an explicit bound on the average rank of $J_d(\Q)$, as well as statistical results on the size of $\#C_d(\Q)$, as $d$ varies through squarefree integers.  
\end{abstract}

\maketitle


\section{Introduction}
We study the arithmetic of an interesting family of genus two curves and their Jacobians.  Combining explicit methods with techniques from geometry of numbers, we give bounds on the average number of rational points in the quadratic twist families of these curves and abelian surfaces, as well as lower bounds on the  size of related Tate-Shafarevich groups.
\subsection{Results}
Our first result is an explicit rational parameterization of the moduli space $\mathcal{H}_3$ of genus two curves $C$, equipped with $\sqrt{3}$-multiplication $\iota \colon \Z[\sqrt{3}] \hookrightarrow \End(J)$ on its Jacobian $J = \Jac(C)$ and a full $\sqrt{3}$-level structure $\epsilon \colon \Z/3\Z \times \mu_3\simeq J[\sqrt 3]$.
Our parameterization uses the polynomials $G_1,G_2,H_1,H_2,\lambda_1,\lambda_2 \in \Z[a,b,c][x]$ given in Section \ref{sec:multsqrt3}; see also \cite{BFScode}. The $G_i(x)$ are cubic, $H_i(x)$ are quadratic, and $\lambda_i \in \Z[a,b,c]$ are constants.      

Let $k$ be a field and assume $\mathrm{char} \, k \neq 2,3$.  For non-zero $(a, b, c) \in k^3$, define the  hyperelliptic curve $C_{a,b,c}$ over $k$ with affine model 
\[y^2 = G_1(x)^2 + \lambda_1H_1(x)^3,\]
obtained by specializing $G_1, \lambda_1,$ and $H_1$ at $(a, b,c)$.  The isomorphism class of $C_{a,b,c}$ depends only on the class $(a \colon b \colon c) \in \PP^2(k)$.  
The ratio  
\[(\minus 3) \cdot \dfrac{G_1^2 + \lambda_1H_1^3}{G_2^2 + \lambda_2H_2^3}\]
is a square in the field of fractions of $\Z[a,b,c][x]$, hence $C_{a,b,c}$ is isomorphic to 
\[\minus3y^2 = G_2(x)^2 + \lambda_2H_2(x)^3.\] 

Outside of the vanishing locus $\Delta$ of the discriminant of the defining sextic (see \eqref{E:discfactors} and \eqref{E:fdisc}), the curve $C = C_{a,b,c}$ is smooth and has genus 2.  Its Jacobian $J = J_{a,b,c}$ is the abelian surface over $k$ parameterizing equivalence classes of degree-zero divisors on $C$.  There are two rational subgroups of order 3 on $J$, defined as follows.  Let $\kappa$ be a canonical divisor on $C$, having degree 2.  Let $x_1$ and $x_1'$ be the roots of $H_1(x)$ in an algebraic closure $k^\alg$ of $k$, and set $y_1=G_1(x_1)$ and $y_1' = G_1(x_1')$. Then the divisor $D_1 = (x_1,y_1) + (x_1',y_1') - \kappa$ has order 3 and represents a point in $J(k)$. It follows that the cyclic subgroup $\langle D_1\rangle$ corresponds to a subgroup scheme $\ZZ/3\ZZ\subset J$ over $k$. Similarly, let $x_2$ and $x_2'$ be the roots of $H_2(x)$, and set 
\[y_2 = \frac{1}{\sqrt{\minus3}}G_2(x_2) \mbox{  and } y_2'= \frac{1}{\sqrt{\minus3}}G_2(x_2').\]  
Then $D_2 = (x_2,y_2) + (x_2',y_2') - \kappa$ has order 3 in $J(k(\sqrt{\minus3}))$, and $\langle D_2\rangle$ corresponds to a subgroup scheme $\mu_3\subset J$ over $k$. We establish the following result.

\begin{thm}\label{model}
The assignment $(a:b:c)\mapsto J_{a,b,c}$ realizes $\PP^2\setminus \Delta$ as an open subset of the moduli space of principally polarized abelian surfaces with real multiplication by $\sqrt{3}$ defined over $k$ and full $\sqrt{3}$-level structure $\ZZ/3\ZZ\times\mu_3\simeq J[\sqrt{3}]$.
\end{thm}

Moduli spaces of abelian surfaces with special endomorphism rings are well-studied, particularly through their closely associated Hilbert modular surfaces. See for instance \cites{vdG, HZ} for a classification of the discriminants that lead to rational surfaces, or the work of Elkies and Kumar \cite{ElkiesKumar}, covering many non-rational Hilbert modular surfaces as well. Results with explicit moduli interpretations, in the sense that a model of the associated universal
genus $2$ curve is provided, are relatively rare.  To our knowledge the parametrization provided here is new.

As a first application of Theorem \ref{model}  we prove a result about the growth of Tate-Shafarevich groups in quadratic twist families of abelian surfaces.  If $A$ is an abelian surface over $\Q$ and $d$ is a squarefree integer, we write $A_d$ for the quadratic twist of $A$ corresponding to the quadratic character $\Gal(\Q(\sqrt{d})/\Q) \to \{\pm1\}$.  For any integer $\ell$, we let $\Sha(A_d)[\ell]$ denote the $\ell$-torsion subgroup of the Tate-Shafarevich group $\Sha(A_d)$.  

Delaunay-type conjectures \cite{delaunay} predict that for any $\ell$, there should be infinitely many $d$ such that $\Sha(A_d)[\ell] \neq 0$; in fact, the natural density of the set $\{d \in \Z \, \colon \Sha(A_d)[\ell] \neq 0\}$ should be positive \cite{BKLOS:sha}*{Conj.\ ~1.1}.
This has been verified for certain twist families of elliptic curves for $\ell = 3$ \cite{BKLOS:sha, ABS}  and $\ell = 2^k$ (\cite{Smith}).

The third author gave examples with $\dim A > 1$ (and $A$ simple) in \cite{shnidman:RM}, by studying quotients of the prime level modular Jacobian $J_0(p)$ with a point of order 3.  For such $A$, one can access properties of their N\'eron models despite not having explicit equations, via Mazur's study of the Eisenstein ideal.  However, it seems likely that there are only finitely many such $A$ of given dimension.              

Our next theorem verifies \cite[Conj.\ 1.1]{BKLOS:sha} for infinitely many abelian surfaces $A$, namely the quotients $A_{a,b,c} := J_{a,b,c}/\langle D_1\rangle$.  In fact, we show that for most of these $A$, a positive proportion of the quadratic twists $A_d$ satisfy $\#\Sha(A_d)[3] \geq 3^r$, for any fixed $r \geq 0$:

\begin{thm}\label{sha}
Fix $r \geq 0$ and order points in $\PP^2(\Z)$ by height.   Then for $100\%$ of rational points $(a \colon b \colon c) \in \mathbb{P}^2(\Z)$, the quadratic twists $A_d$ of the abelian surface $A = A_{a,b,c}$ satisfy
\[\liminf_{X \to \infty} \dfrac{\#\{|d| \leq X \colon \#\Sha(A_d)[3] \geq 3^r\}}{\#\{|d| \leq X\}} > 0.\]
\end{thm}

\begin{rmk}
The same result holds for the surfaces $B_{a,b,c} := J_{a,b,c}/\langle D_2\rangle$ as well. 
\end{rmk}

We say $A$ satisfies Conjecture $(A,3,r)$ if the lim inf in Theorem \ref{sha} is positive. Theorem \ref{sha}  follows from more precise results, which prove Conjecture $(A_{a,b,c},3,r)$ for finitely many $r$, depending on the number and types of primes of bad reduction for $C$; see Theorems \ref{sha6}-\ref{sha2} and \ref{A32}. 
For example, Theorem \ref{A32} implies Conjecture $(A_{a,b,c} \times B_{a,b,c},3,2)$ for all but $142$ of the $219,914$ points $(a \colon b \colon c) \in (\PP^2\backslash \Delta)(\Q)$ of height at most $40$ (Theorem \ref{thm:exp}).

As a second application of our explicit models for the curves $C = C_{a,b,c}$, we bound the average Mordell-Weil rank of the Jacobians $J_d$.  Combining this with the Chabauty-Kim method, we obtain statistical results on the number of rational points on the curves $C_d$ as well.  Our methods apply to any of the curves $C_{a,b,c}$. As an example, we carry out the full analysis for the (simple) Jacobian of minimal root conductor in our family.

\begin{thm}\label{ex1}
The Jacobian $J$ of the curve $C_{1,2,\minus1} \colon y^2 =8x^5 - 3x^4 - 2x^3 - 7x^2 + 4x + 20$
has root conductor $2\cdot 31$.  The average $\Z[\sqrt3]$-rank of $J_d(\Q)$ is at most $1.203$.  At least $23.3\%$ of twists have rank $0$ and at least another $41.6\%$ of twists have $\Z[\sqrt3]$-rank at most $1$.  At least $20.8\%$ of the latter twists have $\Z[\sqrt{3}]$-rank exactly equal to $1$.         
\end{thm}

From these rank results, we deduce uniform bounds on the size of $\#C_d(\Q)$, for many $d$.
This is clear when $\mathrm{rk} \, J_d(\Q) = 0$, because then $C_d(\Q)$ can only contain Weierstrass points (for $d$ large).  When $\mathrm{rk}\, J_d(\Q) > 0$, the method of Chabauty for bounding $\#C_d(\Q)$ does not apply, since the real multiplication implies $\mathrm{rk}_\Z \, J_d(\Q) \geq 2 = g(C_d)$.    Instead, we use recent work of Balakrishnan and Dogra \cite{BD:chabauty} on the Chabauty-Kim method.	

\begin{thm}\label{ex2}
Let $C = C_{1,2,\minus1}$, and consider integers $d \equiv 1 \pmod 3$ ordered by $|d|$.  For at least $23.3\%$ of such $d$, we have $\#C_d(\Q) = 2$, and for at least another $41.6\%$ of $d$ we have $\#C_d(\Q) \leq 154133$.  
\end{thm} 

This seems to be the first non-trivial example of a genus 2 curve $C$ for which $\#C_d(\Q)$ can be uniformly bounded for a majority of $d$.  We say ``non-trivial'' since $C_d$ has rational Weierstrass points (so there are no local obstructions) and admits no maps to elliptic curves (even over $\bar\Q$).  The bound $154133$ is surely far from optimal, but should improve as the Chabauty-Kim method is further refined. 
 
\subsection{Methods and outline}
Theorem \ref{model} is proved in two steps.  In Sections \ref{S:AbelianSurfacesWithZ3Z3} and \ref{sec:multsqrt3}, we  give an explicit parameterization of genus 2 Jacobians $J$ with isotropic $(\Z/3 \times \mu_3)$-level structure, by proving a twisted version of a result of the first two authors and Testa \cite{bruinflynntesta:three}.  Such Jacobians come equipped with a $(3,3)$-isogeny $\pi \colon J \to J'$, and the surface $\mathcal{H}_3$ is the locus where $\pi = \hat\pi$ (and hence $J = J'$ and $\pi = [\sqrt3]$).  This turns out to be a rational surface, and we explicitly write out the universal curve $C_{a,b,c}$ over an open part of the moduli space.

The defining equation for $C_{a,b,c}$ is cumbersome, but most of the geometry is encoded in the ten discriminant factors, which are quite simple and display many symmetries; see (\ref{E:discfactors}) and (\ref{E:fdisc}) in Section~\ref{sec:multsqrt3}.  The non-smooth locus $\Delta = 0$ is a union of genus 0 curves (four lines and six conics) in $\PP^2_\Z$. By explicit computations in the generic fiber, we determine the reduction type of $C_{a,b,c}$ over $\Z_p$ away from the 18 residue disks lying over  the points in $\PP^2(\bar\FF_p)$ where these genus 0 curves intersect (Theorem \ref{semistable reduction}).  
Using the explicit equations, we determine when the 3-torsion points $D_1$ and $D_2$ reduce to the identity component in the N\'eron model of the Jacobian, which allows us to compute the respective Tamagawa ratios $c_p(A)/c_p(J)$ and $c_p(B)/c_p(J)$, listed in Table~\ref{table:tamagawa ratio}. We find that up to some symmetries, the discriminant factors exactly correspond to the possible values of the two Tamagawa ratios.

The Tamagawa ratios are needed to study the Selmer groups of the 3-isogenies $\phi_d \colon J_d \to A_d$ and $\psi_d \colon J_d \to B_d$.  Using an analysis of local Selmer ratios and Poitou-Tate duality, we give a simple criterion in terms of $(a \colon b \colon c)$ for the set $\{d  \colon \#\Sha(A_d)[3] \geq 3^r\}$ to have positive density (Theorem \ref{sha6}); here $A = A_{a,b,c}$.  We show that when there are enough primes dividing specific discriminant factors, then there are many twists $d$ such that $\Sel(\phi_d)$ is large, whereas $\Sel(\psi_d)$ is small on average.  This forces $\Sha(A_d)[3]$ to be large, not just on average, but for many individual $d$ as well.  This method builds on the third author's work \cite{shnidman:RM} and his work with Bhargava, Klagsbrun, and Lemke Oliver  \cite{BKLOS:selmer}.  To conclude that the criterion holds for $100\%$ of $(a \colon b \colon c) \in \PP^2(\Z)$, we invoke Erdos-Kac type results for values of multivariable polynomials \cite{ELS,LOLS}. This proves Theorem \ref{sha}, which can be thought of as a higher-dimensional version of \cite[Thms.\ 1.2-1.4]{BKLOS:sha}.     

In Section \ref{examples}, we prove Theorems \ref{ex1} and \ref{ex2}, and show how to prove explicit statistical results for  $\mathrm{rk}\, J_d(\Q), \, \#C_d(\Q)$, and $\#\Sha(A_d)[3]$, for a fixed curve $C = C_{a,b,c}$ in our family.  The tools here are the geometry-of-numbers techniques of \cite{shnidman:RM, BKLOS:selmer, BKLOS:sha}, results of Balakrishnan-Dogra \cite{BD:chabauty}, and a recent result of Castella-Gross-Li-Skinner \cite{CGLS:IMC} in Iwasawa theory.  

Finally, we recall that our formulas for Tamagawa ratios are only valid away from the 18 points of intersection in $\PP^2(\bar \FF_p)$. Over these 18 degenerate points, the Tamagawa ratios could in principle be computed for specific $(a \colon b \colon c)$ but giving general formulas seems harder.  One issue is that $J_{a,b,c}$ may fail to be semistable at $p$ in this locus, even for $p$ large.  Recall that an abelian variety $A$ over $\Q_p$ with full $\ell$-level structure is automatically semistable if $p \neq \ell > 2$,  by a theorem of Raynaud \cite[Expos\'e IX, 4.7]{SGA7I}.  Similarly, one can show that if $A$ has $\Z[\sqrt{\ell}]$-multiplication and full $\sqrt{\ell}$-level structure, then $J$ is semistable if  $p \neq \ell > 3$.  Our curves show that the inequality $\ell > 3$ is sharp.  Computing Tamagawa ratios when $J$ has additive reduction is quite subtle, since $A$ and $B$ are not themselves Jacobians. Fortunately, additive reduction on $\mathcal{H}_3$ only appears in codimension two, so this is not an issue in the proof of Theorem \ref{sha}.    

\subsection{Acknowledgments}
Thanks go to Michael Stoll for organizing Rational Points 2017, where the authors first discussed this project.  We also thank Levent Alp\"{o}ge, Netan Dogra, Robert Lemke Oliver, Dino Lorenzini, and Daniel Loughran for helpful conversations. The authors used sage \cite{sage} and Magma \cite{magma} for many computations in this paper.

\section{Abelian surfaces with a $(3,3)$-isogeny}
\label{sec:intro}\label{S:AbelianSurfacesWithZ3Z3}

In preparation of the work in Section~\ref{sec:multsqrt3} we first describe a slightly more symmetric situation.
We work over a field $k$ of characteristic not 2 or 3, and consider the group scheme $\Sigma=(\ZZ/3\ZZ)^2$ and its dual $\Sigma^\vee=(\mu_3)^2$. We equip $\Sigma\times \Sigma^\vee$ with the pairing expressing the duality.

Let $\sA(3)$ be the moduli space of principally polarized surfaces $A$, together with an isomorphism $\Sigma\times\Sigma^\vee\to A[3]$, compatible with the Weil pairing $e_3\colon A[3]\times A[3]\to \mu_3$. This moduli space is birational to the Burkhardt quartic threefold in $\PP^4$:
\[ \sB\colon y_0(y_0^3+y_1^3+y_2^3+y_3^3+y_4^3)+3 y_1y_2y_3y_4=0,\]
with automorphism group $\PSp(4,3)$; see \cite{BruinNasserden2018}. The order $3$ cyclic subgroups of $\Sigma\times\Sigma^\vee$ are in bijective correspondence with a collection of $40$ special planes, called \emph{$j$-planes} on $\sB$. Among these are the planes $J_i: y_0=y_i=0$, corresponding to the subgroups of $\Sigma$.

The subgroup of automorphisms of $\sB$ given by
\[H=\{ (y_0:\cdots:y_4)\mapsto (y_0:\zeta^{i_1}y_1:\zeta^{i_2}y_2:\zeta^{i_3}y_3:\zeta^{i_4}y_4): i_1+i_2+i_3+i_3\equiv 0\pmod{3}\}\]
leaves invariant the planes $J_1,J_2,J_3,J_4$, and corresponds to a subgroup of $\PSp(4,3)$ that fixes a maximal isotropic subspace.
We denote by $\sA(\Sigma)$ the moduli space of principally polarized abelian surfaces $A$ together with an injective homomorphism $\Sigma\to A[3]$, such that the Weil pairing pulls back to a trivial pairing.
Then we have
\[\sA(\Sigma)\simeq \sA(3)/H.\]
 This data is equivalent to specifying $A$ together with two $3$-torsion points that pair trivially under the Weil pairing.
Setting $x_i=(y_i/y_0)^3$ for $i=1,\ldots,4$, we obtain an affine birational model of $\sA(\Sigma)$. Using the standard symmetric functions $\sigma_i$ defined by
\[(t+x_1)\cdots(t+x_4)=t^4+\sigma_1t^3+\sigma_2t^2+\sigma_3t+\sigma_4\]
we find the model
\[\sX\colon (1+\sigma_1)^3+27\sigma_4=0.\]
We see that $S_4\simeq \PGL_2(\FF_3)$ acts by permutation on the $x_i$, corresponding to the action of $\Aut(\Sigma)$ on the cyclic subgroups.

We use a construction from  \cite{BruinNasserden2018} that, given a point $y\in \sB$ and a $j$-plane $J$, yields a plane cubic curve
\[E_{J,y}\colon W^3+3\lambda H(X)W+2\lambda G(X)=0,\]
together with a cubic map $E_{J,y}\to\PP^1$ given by projection on the $X$-coordinate. The discriminant of this cubic extension gives rise to a genus $2$ cover
\[C_y\colon Y^2=G_{J,y}(X)^2+\lambda H_{J,y}(X)^3,\]
together with a marked cyclic subgroup of $\Jac(C_y)[3]$.
By \cite[Proposition~2.8]{BruinNasserden2018}, this construction realizes the moduli interpretation of $\sB$.

We perform this construction for the generic point $y$ on $\sB$, and the planes $J_2,J_3$. With an appropriate choice of the coordinate $X$ we find that $C_y$ descends to a curve $C_x$ over $\sX$ of the form
\[C_x\colon Y^2=F_x(X):=G_1^2+\lambda_1H_1^3=G_2^2+\lambda_2H_2^3,\]
with
\[
\begin{aligned}
H_1&=X^2+\frac{\sigma_1+1}{3x_2}X-\frac{x_1x_4}{x_2}\\
\lambda_1&=\frac{x_4(x_1+1)}{x_2}\\
G_1&=\frac{x_4-x_2}{2x_2}X^3-\frac{3x_1x_4+\sigma_1+1}{2x_2}X^2
  -\frac{(3x_1x_4+\sigma_1+1)(\sigma_1+1)}{6x_2^2}X+
\frac{x_1x_4(2x_1x_4+x_3+x_4)}{2x_2^2}
\end{aligned}
\]
and
\[
\begin{aligned}
H_2&=X^2+X+\frac{\sigma_1+1}{3x_2}\\
\lambda_2&=-\frac{x_4^2x_1(x_1+1)}{x_2^2}\\
G_2&=\frac{2x_1x_4+x_2+x_4}{2x_2}X^3-\frac{3x_1x_4+\sigma_1+1}{2x_2}X^2
+\frac{(3x_1x_4+\sigma_1+1)(\sigma_1+1)}{6x_2^2}X+\frac{x_1x_4(x_4-x_3)}{2x_2^2}.
\end{aligned}
\]
The isomorphism $\sA(\Sigma)\simeq \sA(\Sigma^\vee)$ is expressed by the fact that
\[C_x^{(\minus 3)}\colon -3Y^2=F_x(X):=G_1^2+\lambda_1H_1^3=G_2^2+\lambda_2H_2^3\]
has $\Sigma^\vee\subset \Jac(C_x^{(\minus 3)})[3]$.

\begin{lemma} The variety $\sX$ is rational, via the birational map $\sX\to\AA^3$ defined by
\begin{equation}\label{E:rst_parametrization}
(r,s,t)= \left(\frac{\sigma_1 + 1}{3x_2},
\frac{x_2}{x_1x_4 + x_4},
\frac{-x_1x_4}{x_2}
\right).
\end{equation}
\end{lemma}
\begin{proof} See \cite{bruinflynntesta:three}*{Theorem~6} for a derivation of this parametrization.
\end{proof} 

The isogeny $A\mapsto A/\Sigma$ induces a map $\psi\colon \sA(\Sigma)\to\sA(\Sigma^\vee)$. Since $\Sigma^\vee$ is a quadratic twist of $\Sigma$, we have that $\sA(\Sigma)$ is naturally isomorphic to $\sA(\Sigma^\vee)$, so we obtain the analogue of an Atkin-Lehner involution on $\sX$. Translating the formula for $\psi_0'$ in \cite{bruinflynntesta:three}*{Lemma~16}, we find that, where the right-hand-side is defined, we have
\[
\begin{split}\psi(x_1,x_2,x_3,x_4)=\big(h(x_1,x_2,x_3,x_4),
\;h(x_2,x_1,x_3,x_4),
\;h(x_3,x_1,x_2,x_4),
\;h(x_4,x_1,x_2,x_3)\big),\\
\text{ where }
h(x_1,x_2,x_3,x_4)=\frac{3x_1(\sigma_1-x_1+1)-(\sigma_1+1)^2}{3(x_2+1)(x_3+1)(x_4+1)}.
\end{split}\]
It follows that
\[\Jac(C_x)/\Sigma \simeq \Jac(C_{\psi(x)}^{(\minus 3)}).\]

The locus on $\sX$ fixed by $\psi$ is defined by
\begin{equation}\label{E:H3prime}
\sH_3': (\sigma_1+1)^2(\sigma_1-8)-27\sigma_3 = 0.
\end{equation}
Thus for $x\in \sH_3'$, we have that $\Jac(C_x)$ acquires $\sqrt{3}$-multiplication over $k(\sqrt{\minus3})$.

\begin{prop}\label{L:paramg}
The variety $\sH_3'$ is rational, which is demonstrated by the fact that the rational image under \eqref{E:rst_parametrization} is birationally parametrized by restricting the map $\AA^2\to \AA^3$ defined by
	\begin{equation*}
	\begin{split}
	r &= (4 u^2 v^4 + 11 u^2 v^2 + 8 u v^2 + 8 u^2 + 12 u + 4)
	/\bigl( 2  u( u+2) \bigr),\\
	s &= 8  u^2 ( u+2)^3 (4 u v^2 + 5 u + 4)(4 u v^3 + 5 u v -  u + 4 v - 2)\\
	&\quad /\Bigl( (4 u v^3+5 u v+ u+4 v+2)(4 u^2 v^4+11 u^2 v^2+8 u v^2+6 u^2+8 u+4)\\
	&\quad\quad   (16 u^2 v^6+56 u^2 v^4-8 u^2 v^3+32 u v^4+57 u^2 v^2-16 u v^3\\
	&\quad\quad\quad -10 u^2 v+56 u v^2+16 u^2-28 u v+16 v^2+16 u-16 v+4)\\
	&\quad\quad (4 u^2 v^4+11 u^2 v^2+8 u v^2+8 u^2+12 u+4) \Bigr),\\
	t &= -(4 u^2 v^4 + 11 u^2 v^2 + 8 u v^2 + 8 u^2 + 12 u + 4)^2\\
	&\quad (16 u^2 v^6+56 u^2 v^4+32 u v^4+49 u^2 v^2+40 u v^2+8 u^2+16 v^2-4 u-4)\\
	&\quad /\bigl( 8  u^2 (4 u v^2 + 5 u + 4) ( u+2)^3 \bigr),
	\end{split}
	\end{equation*}
\end{prop}

\begin{proof}
The image of $\sH_3'$ under  \eqref{E:rst_parametrization} is given by
\begin{equation}\label{eq:grst}
	\begin{aligned}
		g(r,s,t) = r^6 s^2 t + r^6 s^2 - 3 r^5 s^2 t - 3 r^4 s^2 t^2 
		+ r^3 s^2 t^3 + r^6 s - 3 r^4 s^2 t + 12 r^3 s^2 t^2\\ 
		\quad - 3 r^2 s^2 t^3 - 3 r^5 s - 3 r^4 s t + r^3 s^2 t 
		+ 2 r^3 s t^2 - 3 r^2 s^2 t^2 - 3 r s^2 t^3 + s^2 t^4 + 12 r^3 s t\\ 
		\quad  - 6 r^2 s t^2 + s^2 t^3 + r^3 t - 3 r^2 s t 
		- 3 r s t^2 + 2 s t^3 + r^3 - 3 r^2 t + s t^2 + t^2.
	\end{aligned}
\end{equation}
This is quadratic in~$s$, and the discriminant with respect to~$s$ is
\begin{equation}\label{eq:discrims}
	(r-1)^2(r^2-t)^2(r^6-4r^5-4r^4t+14r^3t-4r^2t-4rt^2+t^2).
\end{equation}
Hence parametrizing $r^6-4r^5-4r^4t+14r^3t-4r^2t-4rt^2+t^2 = \tau^2$ is
equivalent to our original question. This equation is quadratic in $t,\tau$ and with $\sigma=\tau/(2t^2(r-1))$ we find it is birational to a conic bundle over $k(\sigma)$ with a section. This leads to the parametrization given in the lemma.
\end{proof}


\section{Abelian surfaces with multiplication by~$\sqrt{3}$ over the ground field}
\label{sec:multsqrt3}

We observe that for $\Sigma = \ZZ/3\times\mu_3$ we have $\Sigma\simeq \Sigma^\vee$, so if we adjust the constructions in Section~\ref{S:AbelianSurfacesWithZ3Z3} we end up with an involution on $\sA(\Sigma)$, given by $A\mapsto A/\Sigma$. The fixed locus of this involution corresponds to abelian varieties with multiplication by $\sqrt{3}$ \emph{defined over the ground field}. In order to take advantage of the convenient formulas we encountered in Section~\ref{S:AbelianSurfacesWithZ3Z3}, we ``twist'' $\sB$. The resulting variety is in fact isomorphic to $\sB$ again. We set $\rho^2=-3$ and obtain $\sW$ with coordinates $(w_0:w_1:w_2:w_3:w_4)$ by setting
\[(y_0:y_1:y_2:y_3:y_4)=(w_0:w_1+\rho w_4:w_2:w_3:w_1-\rho w_4).\]
We obtain the corresponding twist $\sZ$ of $\sX$ by setting
\begin{equation}\label{E:z_in_x}
(x_1,x_2,x_3,x_4)=(z_1+\rho z_4,z_2,z_3,z_1-\rho z_4),
\end{equation}
and the corresponding degree $27$ map $\sW\to \sZ$ over $k$ is given by
\[z_1=(w_1^3-9w_1w_4^2)/w_0^3, z_2=(w_2/w_0)^3, z_3=(w_3/w_0)^3, z_4=(3w_1^2w_4-3w_4^3)/w_0^3.\]

\begin{lemma}\label{L:parmZ} $\sZ$ is rational via the birational map $\sZ\to\AA^3$ defined by
	\[(u,v,w)=\left(\frac{2z_1+z_2+z_3+1}{3z_2},\frac{2z_1}{z_2},\frac{2z_4}{z_2}\right)\]
\end{lemma}
\begin{proof}Direct computation. \end{proof}

We have two $j$-planes $w_0=w_2=0$ and $w_0=w_3=0$ on $\sW$ as well, which we can use to construct a curve
\[C_w\colon Y^2=G_a(X)^2+\lambda_aH_a(X)^3=-3\big(G_b(X)^2+\lambda_b H_b(X)^3\big).\]
Similar to our approach in Section~\ref{S:AbelianSurfacesWithZ3Z3} we adjust our choice of $X$-coordinate to descend to a curve $C_z$ over $\sZ$.
The answer is less elegant this time, and we give the results in terms of the rational parametrization from Lemma~\ref{L:parmZ}. 
\begin{equation}\label{eq:Cuvw}
\begin{split}
H_a &= -3(x^2+3 u^3+3 u^2 v-3 u x^2+v x^2+3 u^2-3 v^2-9 w^2),\\
\lambda_a &= 4(4 u^3-3 u v^2-9 u w^2+v^3+3 v w^2),\\
G_a &= -18 w (3 u-v-1)^2 x^3\\
 &\quad +27(3 u-v-1) (4 u^3-6 u^2 v+2 u v^2+2 u v-v^2-3 w^2) x^2\\
 &\quad +162u^2 w (3 u-v-1)^2 x\\
 &\quad +108u^6+162 u^5 v+972 u^4 w^2-54 u^3 v^3-324 u^3 v w^2 -324u^5
        -324 u^4 v+189 u^3 v^2+243 u^3 w^2\\
 &\quad +81 u^2 v^3 +243u^2 v w^2-81 u v^4-486 u v^2 w^2-729 u w^4
        +27 v^5 +162v^3 w^2+243 v w^4 +54 u^3 v\\
 &\quad +81 u^2 v^2+243 u^2 w^2 -27v^4-162 v^2 w^2-243 w^4,\\
\end{split}
\end{equation}
\begin{equation*}
\begin{split}
H_b &= (48u^3-48 u^2 v+12 u v^2-12 u^2+12 u v-3 v^2-9 w^2) x^2\\
 &\quad +18w (8 u^3-4 u^2 v-4 u^2+v^2+3 w^2) x\\
 &\quad +108u^3 w^2-36 u^4-36 u^3 v+27 u^2 v^2+81 u^2 w^2 +18u v^3
        +54 u v w^2-9 v^4-54 v^2 w^2-81 w^4,\\
\lambda_b &= (4u^3-v^2-3 w^2) (3 u-v-1)^2
             /(4u^3-3 u v^2-9 u w^2+v^3+3 v w^2)^2,\\
G_b &= \Bigl( -3 (128 u^6-192 u^5 v+96 u^4 v^2-16 u^3 v^3-48 u^5 +72u^4 v
              -52 u^3 v^2-84 u^3 w^2+30 u^2 v^3\\ 
 &\quad\quad  +90u^2 v w^2-12 u v^4
              -36 u v^2 w^2+2 v^5+6 v^3 w^2 -8u^3 v+12 u^2 v^2+36 u^2 w^2\\
 &\quad\quad  -6 u v^3-18 u v w^2 +v^4-9w^4) x^3\\
 &\quad\quad -27w (64 u^6-64 u^5 v+16 u^4 v^2-40 u^5+24 u^4 v -2u^3 v^2
             -6 u^3 w^2+4 u^2 v^3+12 u^2 v w^2\\
 &\quad\quad -2 u v^4 -6u v^2 w^2+8 u^4
             -2 u^2 v^2-6 u^2 w^2-2 u v^3 -6u v w^2+v^4+6 v^2 w^2+9 w^4) x^2\\
 &\quad\quad +27(-96 u^6 w^2+48 u^5 v w^2+16 u^7+8 u^6 v-20 u^5 v^2+12u^5 w^2
             -2 u^4 v^3-6 u^4 v w^2+8 u^3 v^4 \\
 &\quad\quad +36u^3 v^2 w^2+36 u^3 w^4 -2 u^2 v^5-18 u^2 v^3 w^2 
             -36u^2 v w^4-8 u^5 v-4 u^4 v^2 -12 u^4 w^2 +10 u^3 v^3\\
 &\quad\quad +30u^3 v w^2+u^2 v^4-9 u^2 w^4-4 u v^5
             -24 u v^3 w^2 -36u v w^4+v^6+9 v^4 w^2+27 v^2 w^4+27 w^6) x\\
 &\quad\quad +27w (-48 u^6 w^2+24 u^7+24 u^6 v-18 u^5 v^2 -54u^5 w^2
             -12 u^4 v^3-36 u^4 v w^2+6 u^3 v^4\\ 
 &\quad\quad +42u^3 v^2 w^2+72 u^3 w^4
             +8 u^6-18 u^4 v^2-54 u^4 w^2 -2u^3 v^3-6 u^3 v w^2+9 u^2 v^4\\
 &\quad\quad +54 u^2 v^2 w^2+81u^2 w^4-v^6-9 v^4 w^2-27 v^2 w^4-27 w^6)\Bigr)\\
 &\quad\quad\quad  (3u-v-1)/(4 u^3-3 u v^2-9 u w^2+v^3+3 v w^2).
\end{split}
\end{equation*}

The locus $\sH_3\subset \sZ$ corresponding to curves with Jacobians with real multiplication by $\sqrt{3}$ over $k$ is given by the vanishing of
\[\begin{split}
h(u,v,w) &= 16 u^6-12 u^4 v^2-36 u^4 w^2+4 u^3 v^3+12 u^3 v w^2-48 u^4 v
            +48 u^3 v^2+96 u^3 w^2\\
     &\quad -12 u^2 v^3 -36 u^2 v w^2+16 u^3 v -12 u^2 v^2-36 u^2 w^2
            +v^4+6 v^2 w^2+9 w^4.
\end{split}
\]
We introduce notation to more concisely state the expressions in what follows. Writing $\zeta$ for a primitive cube root of unity,  we define some irreducible quadrics and quartics in $\ZZ[a,b,c]$ as follows.
\begin{equation}\label{E:discfactors}
\begin{aligned}
q_3&= Q(a,b,c)=a^2 +a(b-c) + (b-c)^2-3ac,\\
t_3&= Q(a,\zeta b,c)Q(a,{\zeta^2}b,c),\\
q_4&= Q(c,b,a),\\
t_4&= Q(c,\zeta b,a)Q(c,{\zeta^2}b,a).\\
\end{aligned}\end{equation}

\begin{prop}
The surface $\sH_3$ is birational to $\PP^2$ with parameters $(a:b:c)$ via the parametrization
\begin{equation}\label{eq:vwparam}
\begin{split}
u&=\frac{a}{c}\\
v&=
\frac{a (2 a^2 - a b - 8 a c - b^2 - b c + 2 c^2) (a^2 + a b - 4 a c - 2 b^2 + b c + c^2)}{ct_3}\\
w&=\frac{3ab(c-a)(a^2 - 4ac - b^2 + c^2)}{ct_3}
\end{split}
\end{equation}
\end{prop}
\begin{proof} Direct computation.
\end{proof}
When we restrict $C_w$ to $\sH_3$, we find that the substitution
$X=\frac{6ab(a - c)}{c}x - 3\frac{a}{c}$ allows us to remove some common factors and simplify the formulas \eqref{eq:Cuvw}.  We have
\[\begin{aligned}
H_1&=b (a-c) q_4 x^2-q_4 x-1\\
\lambda_1&=4 a c t_3\\
G_1&=b (a - c)^2 (a^2 - 4 a c - b^2 + c^2) q_4^2 x^3
-(a - c) q_4 (a^4 - 4 a^3 b - 8 a^3 c + 3 a^2 b^2
+ 18 a^2 b c + 18 a^2 c^2\\
&+ 2 a b^3 - 9 a b^2 c - 12 a b c^2
- 8 a c^3 - 2 b^4 - b^3 c + 2 b c^3 + c^4) x^2\\
&-q_4 (2 a^3 - 3 a^2 b - 9 a^2 c + 9 a b c + 6 a c^2 + b^3 - c^3) x
-a (a^2 - 2 a b - 4 a c + b^2 + 7 b c + c^2)\\
H_2&=(a - c) (a^3 - 2 a^2 b - 5 a^2 c + a b^2 + 6 a b c + 5 a c^2 - b^2 c - b c^2 - c^3) x^2\\
&+(2 a^2 - 2 a b - 6 a c + 2 b c + c^2) x+1\\
\lambda_2&=-4 q_3 t_3\\
G_2&=(a - c)^2 (2 a^7 - 6 a^6 b - 26 a^6 c + 6 a^5 b^2 + 66 a^5 b c + 126 a^5 c^2 - 4 a^4 b^3
- 54 a^4 b^2 c - 255 a^4 b c^2 \\
&- 278 a^4 c^3 + 6 a^3 b^4 + 28 a^3 b^3 c + 156 a^3 b^2 c^2
+ 396 a^3 b c^3 + 278 a^3 c^4 - 6 a^2 b^5 - 30 a^2 b^4 c\\
&- 54 a^2 b^3 c^2 - 147 a^2 b^2 c^3
- 204 a^2 b c^4 - 126 a^2 c^5 + 2 a b^6 + 18 a b^5 c + 30 a b^4 c^2 + 26 a b^3 c^3\\
& + 18 a b^2 c^4
+ 42 a b c^5 + 26 a c^6 - 2 b^6 c - 3 b^5 c^2 + 3 b^4 c^3 + 4 b^3 c^4 + 3 b^2 c^5 - 3 b c^6 - 2 c^7) x^3
\\
&+3 (a - c) (2 a^6 - 4 a^5 b - 22 a^5 c + 2 a^4 b^2 + 36 a^4 b c + 85 a^4 c^2 - 2 a^3 b^3 - 14 a^3 b^2 c
- 104 a^3 b c^2\\
& - 132 a^3 c^3 + 4 a^2 b^4 + 10 a^2 b^3 c + 27 a^2 b^2 c^2 + 98 a^2 b c^3 + 68 a^2 c^4
- 2 a b^5 - 12 a b^4 c - 10 a b^3 c^2\\
&- 13 a b^2 c^3 - 12 a b c^4 - 14 a c^5 + 2 b^5 c + 2 b^4 c^2
- b^3 c^3 - 2 b^2 c^4 - 2 b c^5 + c^6) x^2\\
&+3 (2 a^5 - 2 a^4 b - 18 a^4 c + 14 a^3 b c + 52 a^3 c^2 - 2 a^2 b^3 - 27 a^2 b c^2 - 49 a^2 c^3
+ 2 a b^4 + 6 a b^3 c\\
& + 13 a b c^3 + 6 a c^4 - 2 b^4 c - b^3 c^2 + 2 b c^4 + c^5) x+
(2 a^3 - 12 a^2 c + 15 a c^2 - 2 b^3 + 2 c^3).
\end{aligned}
\]
We obtain $f_{a,b,c}(x)\in\ZZ[a,b,c][x]$ such that
\begin{equation}\label{E:f_abc}\begin{aligned}
G_1^2+\lambda_1H_1^3&=q_4^2f_{a,b,c}\\
-3(G_2^2+\lambda_2H_2^3)&=(3c)^4f_{a,b,c},
\end{aligned}\end{equation}
and we see that over an open part of $\sH_3$, the universal curve realizing the moduli interpretation is given by
\begin{equation}\label{E:Cabc}
C_{a,b,c}\colon y^2=f_{a,b,c}(x).
\end{equation}
If $\lambda$ is a scalar, then $f_{\lambda a,\lambda b,\lambda c}(x) =\lambda^2f_{a,b,c}(\lambda^2x)$, so the isomorphism class of $C_{a,b,c}$ indeed depends only on $(a \colon b \colon c) \in \PP^2(\Q)$.  
By direct computation we obtain
\begin{equation}\label{E:fdisc}
\disc_x(f_{a,b,c})=2^{12}a^3b^9c^{14}(a-c)^{12}q_3^3t_3^3q_4^4t_4^4.
\end{equation}
In the next section we link the possible reduction types of $C_{a,b,c}$ to the vanishing of the various factors of the discriminant. The following results explain some of the symmetries apparent in the discriminant.

\begin{prop}\label{P:b-isom}
	We have
	\[\C_{a,b,c}\simeq C_{a,\zeta b, c}\]
\end{prop} 
\begin{proof}
We know that $\sX$ has as modular automorphism group $S_4\simeq \PGL_2(\FF_3)$ acting by permutation of $x_1,x_2,x_3,x_4$. By \eqref{E:z_in_x} this gives us the action of $S_4$ on $\sZ$ and therefore on $\sH_3$. From there it is direct computation to check that the map $(a:b:c)\mapsto (a:\zeta b:c)$ is induced by that action.
\end{proof}

Let $J_{a,b,c}$ be the Jacobian of the curve $C_{a,b,c}$. 
\begin{prop}\label{P:richelot}
	The Jacobians $J_{a,b,c}$ and $J_{c,b,a}$ are Richelot-isogenous. As unpolarized abelian surfaces they are isomorphic.
\end{prop}
\begin{proof}
	Since $\End(J_{a,b,c})=\ZZ[\sqrt{3}]$, and $(1+\sqrt{3})(1-\sqrt{3})=-2$, we see $J_{a,b,c}$ admits a $(2,2)$-endormophism $(1 + \sqrt{3})$.  Moreover, since $(1 + \sqrt{3})^2 = (2)$ and the real multiplications  are self-adjoint with respect to the Weil pairing, the kernel is maximal isotropic in $J_{a,b,c}[2]$.  Hence $J_{a,b,c}$ admits a polarized $(2,2)$-isogeny.    By \cite{BruinDoerksen2011}*{Lemma~4.1}, we can obtain a  description of the kernel of the generically unique isogeny by the Galois theory of $f_{a,b,c}$ and by \cite{BruinDoerksen2011}*{Proposition~4.3} we can express the codomain as the Jacobian of a genus 2 curve, which is readily checked to be isomorphic to $C_{c,b,a}$.  Since $1 + \sqrt{3}$ is an endomorphism, we have $J_{c,b,a} \simeq J_{a,b,c}$ as abelian varieties.   
\end{proof}

We see that $J_{a,b,c}$ admits at least two principal polarizations.  If $J_{a,b,c}$ is simple, these are the only ones, up to $\Aut(J_{a,b,c})$-equivalence, by \cite[Cor.\ 2.12]{GGR}.

The final symmetry is less obvious, but comes from the fact that the quadratic twist $J_{\minus3}$ {\it also} has $\Z[\sqrt{3}]$-multiplication and full $\sqrt{3}$-level structure.
\begin{prop}\label{twist}
The $(\minus 3)$-quadratic twist of $C_{a,b,c}$ is isomorphic to 
$C_{q_3,3b(b+a)-q_3,q_4}$ .
\end{prop}
\begin{proof}
The relevant involution on $\sZ$ is $(z_0:z_1:z_2:z_3:z_4)\mapsto(z_0:z_1:z_3:z_2:z_4)$, which is readily transformed into a birational involution on $(u,v,w)$. By direct computation one checks that the given birational involution on $(a:b:c)$ is compatible with that.
\end{proof}
\begin{notation}
	We also consider quadratic twists of $C_{a,b,c}$ and $J_{a,b,c}$, which we denote by $C_{a,b,c;d}$ and $J_{a,b,c;d}$. If $a,b,c$ are clear from the context, we suppress them in our notation and write $C_d=C_{a,b,c;d}$ and $J_d=J_{a,b,c;d}$. 
	
	We write $A_{a,b,c}=J_{a,b,c}/\langle D_1\rangle$ and $B_{a,b,c}=J_{a,b,c}/\langle D_2\rangle$. These abelian surfaces are generally not principally polarized.
	We also consider their quadratic twists $A_d=A_{a,b,c;d}$ and $B_d=B_{a,b,c;d}$.
\end{notation}

\begin{cor}\label{AtwistB}
We have $A_{\minus3} \simeq B_{q_3,3b(b+a)- q_3,q_4}$ and $B_{\minus3} \simeq A_{q_3,3b(b+a)- q_3,q_4}$.  
\end{cor}
\begin{proof}
Indeed, $A$ is the quotient of $J$ by its rational $\sqrt{3}$-torsion point.  The same can be said about the $(\minus3)$-quadratic twist of $B_{q_3,3b(b+a) - q_3,q_4}$, by Proposition \ref{twist}.   
\end{proof}

\section{Reduction types}\label{reduction types}
Let $k$ be a complete discretely valued field of residue characteristic different from $2,3$, with valuation $v\colon k^\times\to \ZZ$, assumed to be normalized, i.e., surjective.
We write $\sO_k$ for the associated complete local ring with maximal ideal $\fp$ and choose a uniformizer $\pi\in\fp$ with $v(\pi)=1$. We write $\kred=\sO_k/\fp$ for the associated residue field.

In this section we determine some of the possible \emph{reduction types} of the curve $C_{a,b,c}$ with $(a:b:c)\in\PP^2(k)$. We take a representative of the point with $a,b,c\in\sO_k$, where $\sO_k$ is the valuation ring in $k$, such that at least one of $a,b,c\in\sO_k^\times$. We assume that the discriminant of $C_{a,b,c}$ does not vanish, so it describes a smooth curve of genus $2$ over $k$. We extend $C_{a,b,c}$ to a model over $\sO_k$ and look at the special fiber.

\subsection{Discriminant factors}\label{subsec:discfactors}
Recall from (\ref{E:discfactors}, \ref{E:fdisc}) that the discriminant of $C_{a,b,c}$ has factors 
\[a,b,c,a-c,q_3, q_4, t_3,t_4.\]
We refer to each of those as \emph{discriminant factors}. Let $h$ be a discriminant factor. We say that $C_{a,b,c}$ has \emph{non-degenerate $h$-reduction} if $v(h(a,b,c))>0$ and $v(h'(a,b,c)) = 0$ for all other discriminant factors $h' \neq h$.
\begin{lemma}\label{L:disc_sing}
The points $(a \colon b \colon c)$ where two or more discriminant factors vanish are
\[
\begin{gathered}
\big\{ (1 : 2 : 1), (1 : 2\zeta : 1), (1:2\zeta^2:1),
	(1 : -1 : 1), (1:-\zeta:1), (1:-\zeta^2:1),\\
	(1:1:0),(1:\zeta:0),(1:\zeta^2:0),
	(0:1:1),(0:\zeta:1),(0:\zeta^2:1),\\
	  (1:0:0),(0 : 1 : 0), (0 : 0 : 1), (1:0:1), (2+\sqrt{3}:0:1), (2-\sqrt{3}:0:1)\big\}
\end{gathered}
\]
\end{lemma}
\begin{proof}
Direct computation.
\end{proof}
\begin{rmk}
These 18 points all lie on the locus $abc(a-c) = 0$. In other words, if two of the higher degree discriminant factors vanish, then so does a linear factor. 
\end{rmk}

As long as $(a\colon b \colon c)$ does not reduce to one of the points in Lemma~\ref{L:disc_sing}, then $C_{a,b,c}$ has either good or non-degenerate $h$-reduction, for some $h$. In the latter case, we will show that the reduction type of $C_{a,b,c}$ can then be classified by the integer $m=v(h(a,b,c))$.

\subsection{Non-degenerate reduction}

Recall that $C/k$ is said to have \emph{semistable reduction} if there is a regular model $\mathcal{C}$ for $C$, flat and proper over $\sO_k$, such that the special fiber is a reduced normal crossing divisor in $\mathcal{C}$.  In other words, the singularities on the special fiber are ordinary double points.  Semistable reduction of $C$ over $k$ implies semistable reduction of the Jacobian $J = J(C)$ over $k$.  If $\mathcal{J} \to \Spec\sO_k$ is the N\'eron model of $J$, this means that the connected component $(\mathcal{J} \otimes \kred)^0$ of the special fiber is an extension of an abelian variety by a torus.  

For abelian surfaces with real multiplication, there are only two options for $(\mathcal{J} \otimes \kred)^0$: it is either an abelian variety (if $J$ has good reduction) or it is a torus of rank two.  In the latter case, we say $J$ has \emph{toric reduction}.
\begin{lemma}\label{torus classification}
Let $k$ be a field of characteristic different from $2,3$. Then the rank two tori with real multiplication by $\ZZ[\sqrt{3}]$ are classified by $\cH^1(k,\mu_2)=k^\times/k ^{\times 2}$. The associated class is called the \emph{torus character}.
\end{lemma}
\begin{proof} 
Rank two tori over $k$ with an action by $\Z[\sqrt{3}]$ are parameterized by $\Z[\sqrt3]$-modules of rank 1 endowed with a continuous action of $\Gal(k^\sep/k)$.  Since $\Z[\sqrt{3}]$ has class number 1, all such modules are twists of $\Z[\sqrt{3}]$. The torsion part of $\Aut(\Z[\sqrt{3}])=\Z[\sqrt{3}]^\times=\ZZ\times\mu_2$ is $\mu_2$, so these twists are classified by $\cH^1(k,\mu_2)$, which by Kummer theory is just $k^\times/ k^{\times2}$.
\end{proof}
 
If $J$ has toric reduction at $p$ and $(\mathcal{J} \otimes \kred)^0 \simeq  \mathbb{G}_\mathrm{m}^2$, we say $J$ has \emph{split toric reduction}. Otherwise, it has \emph{non-split toric reduction}.
The torus character can be read off from the tangent cone of one of the nodes of $C$; the real multiplication forces the quadratic characters for each of these nodes to agree.

The other bit of data we will need is the component group $\Phi$, which is the zero-dimensional group scheme over $\kred$ defined by exact sequence
\begin{equation}\label{E:Phi_def}
0\to (\sJ\otimes \kred)^0\to \sJ\otimes \kred \to \Phi\to 0
\end{equation}
If $\sJ$ has split toric reduction, then $\Phi$ is a constant \'etale group scheme. If $\sJ$ has toric reduction (not necessarily split) then by Lemma~\ref{torus classification}, we see that $\Phi$ is constant after at most a quadratic extension determined by the torus character.

Possible reduction types of hyperelliptic curves, and in particular of curves of genus $2$, have been classified \cite{namikawa-ueno}. The relevant types for us are the following.

\begin{prop}\label{P:I_m1-m2-m3}
	Let $f(x)\in \sO[x]$ be a sextic polynomial with unit leading coefficient. Suppose $f(x)=f_1(x)f_2(x)f_3(x)$ with $f_i\in\sO[x]$, such that the reductions of the $f_i$ are pairwise coprime and have at worst double roots. Then the model
	\[\sC\colon y^2=f_1f_2f_3\]
	has semistable reduction of type $[I_{m_1-m_2-m_3}]$, where $m_i=v(\disc f_i)$.
	
	For type $[I_{m-m-m}]$ with $m>0$, we have $\Phi\otimes \kred^\mathrm{alg} \simeq \ZZ/m\ZZ\times \ZZ/3m\ZZ$.  For type $[I_{m_1-m_2-0}]$, we have $\Phi\otimes \kred^\mathrm{alg}=\ZZ/m_1\ZZ\times \ZZ/m_2\ZZ$
\end{prop}

\begin{proof}
	To ease notation we set $i=1$.
	The relevant observation here is that after a coordinate change, we can assume that $f_1=c_0x^2+\pi^{m_1}(c_1+c_2x+c_3x^2)$, so that we have
	$\disc(f_1)=\pi^{m_1}(4c_0c_3+\pi^{m_1}(c_2^2-4c_1c_3))$. It follows that the associated point that is singular in the special fibre, locally looks like
	\[y^2-c_0f_2(0)f_3(0)x^2=(4c_0c_3)\pi^{m_1}.\]
    This is an $A_{(m_1-1)}$-singularity, leading to the reduction type stated.
	For the structure of the component groups, see \cite[\S9]{Muller-Stoll} or \cite[9.6.10]{BLR}. Note that the tangent cone at the reduced point is split precisely when $c_0f_2(0)f_3(0)$ is a square in $\kred$.
\end{proof}
\begin{thm}\label{semistable reduction}
Fix $(a\colon b\colon c) \in \mathbb{P}^2(\sO_k)$.  Suppose that $v(h(a,b,c))>0$ for exactly one discriminant factor $h(a,b,c)\neq a-c$, and set $m=v(h(a,b,c))$.
Then $C_{a,b,c}$ has semistable reduction type as in Table $\ref{non-degenerate reduction}$.  Moreover, $J_{a,b,c}$ is semistable and has toric reduction.
\end{thm}
\begin{table}
\[ \renewcommand{\arraystretch}{1.2}
   \begin{array}{|r|c|c|c|} \hline
      h(a,b,c) & \text{reduction type} & \text{torus character} & \Phi \otimes \kred^\mathrm{alg}  \\\hline
      a & [I_{m-m-m}]   & 1        & \Z/m\Z \times \Z/3m\Z  \\
      b & [I_{3m-6m-0}]  & 1          & \Z/3m\Z \times \Z/6m\Z   \\
      c & [I_{m-3m-0}] & 1      & \Z/m\Z \times \Z/3m\Z  \\
     Q(a,\zeta^i b,c) & [I_{m-m-m}]  & -3         & \Z/m\Z \times \Z/3m\Z  \\
     Q(c,\zeta^i b,a) & [I_{m-3m-0}] & -3           & \Z/m\Z \times \Z/3m\Z  \\
\hline
    \end{array}
\]
\caption{Local data at primes of non-degenerate bad reduction.  Here, $m = v(h(a,b,c))$.} \label{non-degenerate reduction}
\end{table}
\begin{proof}
For each of the choices for $h(a,b,c)$, the proof follows the same pattern. 
For $h(a,b,c)=a$, we parametrize the locus $h(a,b,c)=0$ by setting $(a:b:c)=(0:s:t)$. We find
\[f_{0,s,t}=(s^3-t^3)^2X^2(stX+1)^2(t(s+t)X+1)^2.\]
So as long as $v(a)>0$ and $(a,b,c)$ does not reduce to one of the points in Lemma~\ref{L:disc_sing}, then we see that $f_{a,b,c}$ has unit leading coefficient and reduces to a product of pairwise coprime squares. These lift to a quadratic factorization of $f_{a:b:c}$ over $\sO_k$, so Proposition~\ref{P:I_m1-m2-m3} applies. Computation shows that we have type $[I_{m-m-m}]$ with $m=v(h(a,b,c))$.
We see that the tangent cones at the nodes of the model in reduction are split, so the torus character is trivial.

For $h(a,b,c)=b$ we set $(a:b:c)=(s:0:t)$. We consider $X^6f_{a,b,c}(1/X)$ and find in reduction:
\[X^2(X + s^2 - 4st + t^2)^2(s(s-4t)X^2+s(s-3t)(s^2 - 4st + t^2)X+(s-t)^2(s^2 - 4st + t^2)^2).
\]
This yields reduction type $[I_{3m-6m-0}]$ with $m=v(\disc(f))$ for nondegenerate reduction, with trivial torus character.

For $h(a,b,c)=c$ we find that $f_{a,b,c}(X)$ does not exhibit stable reduction. We change to
\[\tilde{f}_{a,b,c}=\frac{a^6(a-b)^6X^6}{c^4}f_{a,b,c}\left(\frac{c}{X}-\frac{1}{a(a-b)}\right).\]
Setting $(a:b:c)=(s:t:0)$, we find
\[\tilde{f}_{s,t,0}=(s-t)^3 X^2((3s-t)X+s^2(s-t))^2q(X)\]
with $q(x)\in \sO[s,t][X]$ separable. Further computation yields reduction type $[I_{m-3m-0}]$ and trivial torus character. 

For $h(a,b,c)=q_3 = Q(a,b,c)$ we use the parametrization $(a:b:c)=(3t^2:t^2+4st+s^2:s^2+st+t^2)$. We proceed exactly as above and find reduction type $[I_{m-m-m}]$. We see that the tangent cones split if $-3$ is a square, so the torus character is the one associated to $k(\sqrt{\minus3})$. By Proposition~\ref{P:b-isom}, the same result holds for $h(a,b,c)=t_3=Q(a,\zeta b,c)Q(a,\zeta^2b,c)$, although for nondegenerate $t_3$-reduction, we need $\sqrt{\minus3}\in k$ already.

For $h(a,b,c)=Q(c,b,a)$ we use $(a:b:c)=(s^2+st+t^2:t^2+4st+s^2:3t^2)$ and find reduction type $[I_{m-3m-0}]$ with torus character associated to $k(\sqrt{\minus3})$.
\end{proof}

Finally, we treat the remaining discriminant factor $a-c$.
\begin{thm}
Suppose that $v(a-c)>0$ and that $v(h(a,b,c))=0$ for the other discriminant factors. Then $J_{a,b,c}$ has good reduction. In fact, $J_{a,b,c}$ reduces to a product of elliptic curves. 
\end{thm}

\begin{proof}
We use \cite[Thm.\ 1.8(5)]{DDMM}, which gives a necessary and sufficient condition for $J$ to have good reduction in terms of a \emph{cluster picture} $\sR$ determined by the six roots of $f(x)$. 

Write $k^\alg$ for the algebraic closure of $k$ and write $\sO^\alg$ for its ring of integers over $\sO$, with $v\colon (k^\alg)^\times\to \QQ$ the extension of the valuation. We endow $\PP^1(k^\alg)$ with a logarithmic distance induced by $v$ on the standard covering by two copies of $\sO^\alg$ of $\PP^1(k^\alg)$. 
For a hyperelliptic curve $y^2=f(x)$, we identify the roots of 
$f$ with a set of six points $\sR=\{\alpha_1,\ldots,\alpha_6\}\subset \PP^1(k^\alg)$. A \emph{subcluster of depth} $n$ and center $\beta$ is a subset
\[\fs=\{\alpha\in \sR: d(\alpha,\beta)\leq n\}.\]
The criterion \cite[Thm.\ 1.8(5)]{DDMM} for good reduction of $\Jac(C)$ is that the splitting field of $f$ is unramified, that all proper subclusters of $\sR$ have odd cardinality, and that all clusters of cardinality at least $3$ have even depth.

We parametrize $a-c=0$ by setting $(a:b:c)=(s:t:s)$. We obtain
\[f_{s,t,s}(x)=((8s^3 - t^3)x -3s)(s^2(4s^3 + t^3)x^2 -(4s^3 + t^3)x + s),\]
We particularly see that $f_{a,b,c}$ has three roots with negative valuation, so the naive model does not have stable reduction. We also see that
\[f_{a,b,c}(x)=c(x)q(x)\ell(x),\]
where $c(x),q(x),\ell(x)\in \sO_k[x]$ have degrees $3,2,1$ respectively. Let us write $\alpha_1,\alpha_2,\alpha_3$ for the points corresponding to the roots of $c(x)$, and $\alpha_4,\alpha_5$ for the roots of $q(x)$, and $\alpha_6$ for the root of $\ell(x)$.  The discriminant of $q(x)\ell(x)$ is a unit, so each of $\{\alpha_4\},\{\alpha_5\},\{\alpha_6\}$ are children (maximal proper subclusters) of $\sR$.

We determine the Galois group of $f_{a,b,c}$. From Proposition~\ref{P:richelot} we know that $J_{a,b,c}$ admits a Richelot isogeny, so by \cite{BruinDoerksen2011} we know that the Galois group acts through the stabilizer of $\{\{1,2\},\{3,4\},\{5,6\}\}$ in $S_6$.  This is not possible if $f(x)$ has both a linear and irreducible cubic factor, so $c(x)$ must be reducible. If we write $c(x) = q_1(x)\ell_1(x)$, then 
\[v(\disc(q_1)) \stackrel{\pmod2}{\equiv} v(\disc(c)) = v(\disc(f)) \stackrel{(\ref{E:fdisc})}{=} 12m.\]
It follows that the splitting field of $q_1(x)$ is unramified.  Since this is true for $q(x)$ as well, the splitting field of $f(x) = q_1(x)\ell_1(x)q(x)\ell(x)$ is unramified, of degree at most 2.    

We see that $\{\alpha_4\},\{\alpha_5\},\{\alpha_6\}$ are subclusters of depth $1$, and $\fs=\{\alpha_1,\alpha_2,\alpha_3\}$ is a subcluster of depth at least $1$ with center $\infty$. Since $\sR$ has at least three children of odd cardinality, $J$ does not have potentially totally toric reduction \cite[Thm.\ 1.8(8)]{DDMM}. Since $J$ has real multiplication, it must therefore have potentially good reduction.  By \cite[Thm.\ 1.8(6)]{DDMM}, all proper subclusters $\fs' \subsetneq \mathcal{R}$ have odd cardinality. This implies that $\alpha_1,\alpha_2,\alpha_3$ are pairwise equidistant. From \eqref{E:fdisc} it follows that
\[2(d(\alpha_1,\alpha_2)+d(\alpha_1,\alpha_3)+d(\alpha_2,\alpha_3))=v(\disc(f_{a,b,c}))=12m,\]
so we see that the depth of $\fs$ is $2m$, which is indeed even.

Finally, by \cite[Thm.\ 1.8(2)]{DDMM}, the curve $C$ has bad reduction.  Since its Jacobian has good reduction, it follows that its Jacobian reduces to a product of elliptic curves.
\end{proof}


\section{Ratios of Tamagawa numbers}\label{ratios}
In this section, we compute certain ratios of Tamagawa numbers associated to the curve $C = C_{a,b,c}$ over $\Q_p$, where $p$ is a prime of bad reduction.  To begin, we continue with the setting and notation of Section \ref{reduction types}, working over a more general local field $k$. 

If $A/k$ is an abelian variety, its Tamagawa number is $c(A)=\#\Phi_A(\bar k)$,
where $\Phi_A$ is the component group of $A$, as defined in \eqref{E:Phi_def}.  The specialization map 
\[\mathrm{sp} \colon A(k) \longrightarrow \Phi_A(\bar k)\] is defined by extending  $P\in A(k)$ to $\mathcal{A}(\mathcal{O}_k)$ and then specializing to the special fiber of the N\'eron model $\mathcal{A}$. 

Now let $J = J_{a,b,c} = \mathrm{Jac}(C)$ and let $k$ be a finite extension of $\Q_p$ for some prime $p$. By construction we have $\langle D_1\rangle \times \langle D_2 \rangle \simeq (\ZZ/3\ZZ)\times\mu_3\subset J[3]$. Write
\[A=J/ \langle D_1\rangle \text{ and } B=J/\langle D_2 \rangle\]
for the corresponding quotients over $k$.  

The goal of this section is to compute the ratios $c(A)/c(J)$ and $c(B)/c(J)$.  Depending on the reduction type of $C$, this could be a rather subtle question.  For the proof of Theorem~\ref{sha}, it will be enough to consider cases of semistable reduction.  With this in mind, we give a general criterion for computing Tamagawa ratios of $\ell$-isogenous abelian varieties with totally split toric reduction.

\begin{prop}\label{split toric formula}
Let $J$ be an abelian variety over $k$ with split totally toric reduction.  Let $\langle P\rangle \subset J[\ell]$ be a $k$-rational subgroup of prime order $\ell \neq p$, and let $A = J/\langle P\rangle$.  Assume $P \in J(k^\alg)$ is defined over an unramified extension $F/k$.  Then  
\[c(A)/c(J) = \begin{cases}
\ell & \mbox{if } \mathrm{sp}_F(P) = 0\\
1/\ell & \mbox{ otherwise}.
\end{cases}
\]  
\end{prop}

\begin{proof}
Tamagawa numbers of abelian varieties with split totally toric reduction are unchanged by unramified base change, so we may assume $F = k$.  By the theory of $p$-adic uniformization there is an analytic isomorphism $J \simeq T/\Lambda$, where $T = \mathbb{G}_{m,k}^g$ and $\Lambda = \langle q_1,\ldots, q_g\rangle \subset k^{\times g}$ is a full rank multiplicative lattice.    

If $\mathrm{sp}(P) \neq 0$, then $c(A)/c(J) = 1/\ell$; the proof is exactly as in \cite[Prop.\ 6.6]{shnidman:RM}.  Note that the condition $\mathrm{sp}(P) \neq 0$ is equivalent to saying that we may choose the basis $\{q_i\}$ of $\lambda$ such that $A \simeq T/\langle q_1^{1/\ell},q_2, \ldots, q_g\rangle$.  

If $\mathrm{sp}(P) = 0$, then $P$ belongs to the subgroup 
\[\mu_\ell^g \hookrightarrow (T/\Lambda)[\ell] \simeq J[\ell].\]
Note that the quotient $J/\mu_\ell^g$ is isomorphic $T/\langle q_1^\ell, \ldots, q_g^\ell\rangle$ via the map
\[(x_1, \ldots x_g) \mapsto (x_1^\ell, x_2^\ell, \ldots, x_g^\ell).\]  
In the other direction, the natural $\ell^g$-isogeny 
\[J/\mu_\ell^g \simeq T/\langle q_1^\ell, \ldots, q_g^\ell\rangle \to T/\langle q_1, \ldots, q_g\rangle = J\] 
factors as a sequence of $g$ different $\ell$-isogenies whose kernels map non-trivially into the component group.  So $c(J/\mu_\ell^g) = \ell^g c(J)$, by what we've already proven.  To prove that $c(A)/c(J) = \ell$ in this case, we use the fact that $c(A)/c(J)$ is either $\ell$ or $1/\ell$, for any $\ell$-isogeny $J \to A$ \cite[Lem.\ 2.3]{K-P}.    Since the total effect of the $g$ different $\ell$-isogenies is to multiply the Tamagawa number by $\ell^g$, it follows that for each subgroup $\langle P \rangle \subset \mu_\ell^g$, we have $c(J/\langle P \rangle) = \ell c(J)$, since we may choose $J \to J/\langle P\rangle$ to be the first isogeny in the chain.  
\end{proof}

We also deduce a formula for the Tamagawa ratio of the quadratic twists $J_d$, for $d \in k^{\times}$.  

\begin{cor}\label{non-split toric formula}
Let $J$ be as above and suppose $\mathrm{char} \, \bar k$ and  $\ell$ are distinct odd primes. If $d \notin k^{\times 2}$, then $c(A_d)/c(J_d) = 1$. 
\end{cor}

\begin{proof}
Let $J_L$ be the base change of $J$ to $L:=k(\sqrt d)$. 
The corollary follows Proposition \ref{split toric formula} and the formula $c(A_d)/c(J_d) = c(A_L)/c(J_L) \cdot c(J)/c(A)$ \cite[Lem.\ 4.6]{D-D}.  
\end{proof}

Let $\zeta_\ell$ be a primitive $\ell$-th root of unity. 
\begin{cor}\label{p=2}
Suppose $\zeta_\ell \notin k$ and $\mathrm{char} \, \bar k \neq \ell$. Suppose $J$ has split totally toric reduction and rational subgroups $H_1 \simeq \Z/\ell\Z$ and $H_2 \simeq \mu_\ell$.  Let $A = J/H_1$ and $B = J/H_2$.  Then
\[c(A)/c(J) = \frac{1}{\ell} \hspace{3mm} \mbox{ and } \hspace{3mm} c(B)/c(J) = \ell.\]
\end{cor} 
\begin{proof}
Since $J$ has split toric reduction, we have an exact sequence of group schemes
\[0\to \mu_\ell^g \to J[\ell] \to (\Z/\ell\Z)^g \to 0.\]
Since $\zeta \notin k$, we have $\mu_\ell \not\simeq \Z/\ell\Z$.  Hence $H_2$ lies in $\mu_\ell^g$ whereas $H_1$ does not.  In other words, $H_2$ specializes to the identity component and $H_1$ does not. By Proposition \ref{split toric formula}, we have the desired formula for the Tamagawa ratios.
\end{proof}
\begin{rmk}
In Section \ref{examples}, we use Corollary \ref{p=2} when $k = \Q_2$ and $\ell = 3$. This allows us to avoid delicate  computations of integral models over $\Z_2$.   
\end{rmk}

Next, we apply Proposition \ref{split toric formula} to the Jacobian $J_{a,b,c}$ in various cases of non-degenerate reduction (and hence semistable reduction, by Theorem \ref{semistable reduction}).  For simplicity we take $k = \Q_p$, so that $a,b,c \in \Z_p$.  Thus $C_{a,b,c}$ has $h$-reduction over $\Q_p$ if $p \mid h(a,b,c)$ and $p \nmid h'(a,b,c)$ for all other discriminant factors $h' \neq h$.

The explicit models (\ref{E:f_abc}) for $C_{a,b,c} \colon y^2 = f_{a,b,c}(x)$ differ slightly from those in the introduction, so let us describe the divisor classes $D_1$ and $D_2$ in $J[3]$ in this model.  Write $D_1 = P_1 +P'_1 - \kappa$, with $\kappa$ an effective canonical divisor and $P_1,P'_1$ points over the algebraic closure.  If $P_1 = (x_1,y_1)$ and $P'_1 = (x'_1, y'_1)$,
then $x_1$ and $x'_1$ are the two roots of $H_1(x)$, and $y_1 = q_4^{-1}G_1(x_1)$ and $y'_1 = q_4^{-1}G_1(x'_1)$. Similarly, if $D_2 = P_2 + P'_2 - \kappa = (x_2,y_2) + (x'_2,y'_2) - \kappa$, then $x_2$ and $x'_2$ are the roots of $H_2(x)$, while $y_2 = (3c^2\sqrt{-3})^{-1}G_2(x_2)$ and $y_2' = (3c^2\sqrt{-3})^{-1}G_2(x_2')$.  Here, $\sqrt{-3}$ is a fixed choice of square root.    

\subsection{Tri-nodal reduction}\label{trinode}
Assume first that $J$ has $a$-reduction, and let $m = v_p(a) > 0$.  Then $J$ has split toric reduction over $\Q_p$, by Table \ref{non-degenerate reduction}.  Since $q_4 \in \Z_p^\times$, an integral model  for $C$ over $\Z_p$ is
\[\Cnaive \colon y^2 = q_4^2f_{a,b,c}(x) = G_1(x)^2 + \lambda_1H_1(x)^3.\] 
Let $\Cmin/\Z_{p}$ be the minimal regular model.  Since $a \mid \lambda_1$, the special fiber of $\Cnaive$ has two components $F_1$ and $F_2$ with equations $y = G_1(x)$ and $y = -G_1(x)$, and which intersect at the three nodes.  Then $\Cmin$ is obtained from $\Cnaive$ by replacing each of the three nodes with a chain of $m-1$ rational curves $X_1,\ldots, X_{m-1}$, $Y_1,\ldots, Y_{m-1}$, and $Z_1, \ldots, Z_{m-1}$ connecting $F_1$ to $F_2$; see \cite[Fig.\ 3]{Muller-Stoll}.  
If $m = 1$, the two components $F_1$ and $F_2$ intersect each other transversally in the three nodes.   

\begin{rmk}
When $C$ has semistable reduction, Raynaud's theorem \cite[9.5 Thm.\ 5]{BLR} says that the N\'eron model of $J$ represents $\mathrm{Pic}^0_{\Cmin/\Z_{p}}$.  This gives a convenient way to evaluate the specialization map $\mathrm{sp} \colon J(\Q_p) \to \Phi_J(\FF_p)$ on a divisor class $D$.  Namely, write $D = \sum n_i P_i$, with each $P_i$ point reducing to a unique component $C_i$ of $\Cmin_{\FF_p}$.  After identifying $\Phi_J(\FF_p)$ with the divisor class group of the dual graph of $\Cmin$ \cite[9.6 Thm.\ 1]{BLR}, we have $\mathrm{sp}(D) = \sum n_i C_i$.
\end{rmk}

\begin{prop}\label{tra}
If $p  > 3$ and $J$ has $a$-reduction over $\Q_p$, then 
\[c(A)/c(J )=\frac13 \hspace{3mm}\mbox{ and } \hspace{3mm} c(B)/c(J )= 3.\]
\end{prop}
\begin{proof}
For polynomials $F,G \in \Q[x]$, we denote the resultant by $\mathrm{Res}_x(F,G)$.  Recall $D_1 = P_1 + P_1' - K$, with $P_1 = (x_1,y_1)$ and $P'_1 = (x'_1,y'_1)$.  Then $y_1y'_1 = \mathrm{Res}_x(H_1,G_1)h_1^{-3}$, where $h_1$ is the leading coefficient of $H_1(x)$.  Over $\FF_p$ (so that $a \equiv 0$), we compute $h_1 \not\equiv 0$, and it follows that $y_1$ and $y'_1$ are integral over $\Z_p$.  Again over $\FF_p$, 
\[y_1y'_1 = \mathrm{Res}_x(H_1,G_1)h_1^{-3} \equiv   c^2 t_4 \not\equiv 0,\]
which shows that $y_1$ and $y'_1$ are in fact $p$-adic units.  

This means that $P_1$ and $P_1'$ reduce to non-nodes on the special fiber of $\Cnaive$, and hence reduce to either $F_1$ or $F_2$ on $\Cmin$.  By their very definition, they reduce to the {\it same} component, so without loss of generality we may assume they reduce to $F_1$.  We may represent $K$ by a degree two divisor which does not intersect the nodes.  Then $K$ reduces to $[F_1 + F_2]$ since it is an orbit of the hyperelliptic involution.  We deduce that 
\[\mathrm{sp}(D_1) = [2F_1 - F_1 - F_2] = [F_1 - F_2].\]
By Proposition \ref{split toric formula}, to prove that $c(A)/c(J) = 1/3$ it is enough to check that $[F_1- F_2]$ specializes to a non-trivial element of the component group $\Phi_J(\FF_p)$.
Recall from \cite[9.6.10]{BLR} that $\Phi_J(\FF_p)$ is isomorphic to the group of degree-zero divisors on the dual graph of the special fiber, modulo the subgroup generated by the relations 
\begin{equation}\label{graph relations}
\deg(F) [F] - \sum_{F'\neq F} (F' \cdot F) [F']= 0,
\end{equation}
for any component $F$.
    
Suppose that $[F_1] - [F_2] = 0$ in the component group.
It follows from the structure of the special fiber and (\ref{graph relations}) that the component group is generated by pairwise differences amongst the four components $[X_{m-1}], [Y_{m-1}], [Z_{m-1}]$ and $[F_2]$.  By Theorem \ref{semistable reduction}, the component group is isomorphic to $\Z/m\Z \times \Z/3m\Z$.  We will reach a contradiction by showing that these pairwise differences are $m$-torsion in the component group; since these generate the group, this would mean that there is no point of order $3m$.

For this, we compute
\begin{align*}[F_2] = [F_1] &= 2[X_1] - [X_2] \\
&= 4[X_2] - 2[X_3] - [X_2] = 3[X_2] - 2[X_3]\\
&= 6[X_3] - 3[X_4] - 2[X_3]= 4[X_3] - 3[X_4]\\
&\vdots\\
& = (m-1)[X_{m-2}] - (m-2)[X_{m-1}]\\
&= m[X_{m-1}] - (m-1)[F_2],  
\end{align*}
and so $m[F_2] - m[X_{m-1}] = 0$. Simlarly, $m([F_2] - [Y_{m-1}]) = m([F_2] - [Z_{m-1}]) = 0$.  It follows that the component group is generated by $m$-torsion elements, and hence is itself $m$-torsion, which is a contradiction.  Thus $[F_1] - [F_2]$ is not the identity in the component group, and by Proposition \ref{split toric formula}, $c(A)/c(J) = 1/3$.  

Next we consider the divisor class $D_2 = P_2 + P_2' - K = (x_2,y_2) + (x_2',y_2') - K$.  Over $\FF_p$,  
\[\Res_x(G_2,H_2)h_2^{-3} \equiv 3^3  c^6 q_3q_4^{-1},\]
which is non-zero. This shows that $P_2$ and $P_2'$ reduce to smooth points on the (naive) special fiber.  The question is whether they reduce to the same component.  We compute
\[\frac{1}{\minus27c^4}\Res_x(H_2,G_2)\equiv -q_4^{-2}\Res_x(H_2,G_1),\]
which means that $y_2y_2' \equiv -G_1(x_2)G_1(x_2')$. It follows that $P_2$ and $P_2'$ reduce to different components, and hence $\mathrm{sp}(D_2) = [F_1+F_2] - [F_1 + F_2] = 0$. By Proposition \ref{split toric formula}, we have $c(B)/c(J) = 3$.           
\end{proof}

When $C$ has $q_3$-reduction, $J$ has toric reduction with torus character $-3$, by Table \ref{non-degenerate reduction}. To see interesting Tamagawa ratios, we twist by this character. 
\begin{prop}\label{trq3}
If $p  > 3$ and $J$ has $q_3$-reduction, 
then  
\[c(A_{\minus3})/c(J_{\minus3}) =3 \hspace{3mm} \mbox{ and } \hspace{3mm} c(B_{\minus3})/c(J_{\minus3})= \frac{1}{3}.\]
\end{prop}
\begin{proof}
The abelian surface $J_{\minus3}$ has split toric reduction, so we may apply Proposition \ref{split toric formula}.  The proof is then very similar to the case of $a$-reduction, except now $p \mid \lambda_2$, so the roles of $D_1$ and $D_2$ are reversed.  We omit the details.
\end{proof}      

Next we consider non-degenerate $t_3$-reduction. If $t_3 = q_3(a,\zeta b,c)q_3(a,\zeta^2b,c)$ vanishes in $\FF_p$, then we must have $\zeta \in \FF_p$. Indeed, non-degenerate reduction implies $(a \colon b \colon c)$ lies on the smooth locus of the union of the two conics.  Thus the equation must factor over $\FF_p$ and hence $\zeta \in \FF_p$.  
So the torus character is trivial and there is no need to twist.  
\begin{prop}\label{trq3}
Let $p  > 3$, and assume $J = J_{a,b,c}$ has $t_3$-reduction.  Then 
\[c(A)/c(J )=\frac{1}{3}  = c(B)/c(J ).\]
\end{prop}
\begin{proof}
In this case $p$ divides {\it both} $\lambda_1$ and $\lambda_2$.  By a similar computation, we find that both $D_1$ and $D_2$ specialize to non-zero elements of the component group. 
\end{proof}

\subsection{$b$-reduction}
Assume now that $J$ has $b$-reduction, and $b$ has valuation $m > 0$.  Let $\Cnaive/\Z_p$ be the curve defined by $y^2 = f_{a,b,c}(x)$ and let $\Cmin/\Z_p$ be the minimal regular model.  The special fiber of $\Cnaive$ is connected, with a single component $F$ intersecting itself at two nodes.  The special fiber of $\Cmin$ is obtained from that of $\Cnaive$ by replacing one node with a chain of $3m-1$ rational curves $X_1,\ldots, X_{3m-1}$, and replacing the other node with a chain of $6m-1$ rational curves $Y_1,\ldots, Y_{6m-1}$.  Each of the two chains connect $F$ to itself, as in \cite[Fig.\ 2]{Muller-Stoll}.
 
From the relations (\ref{graph relations}), we deduce that $\Phi$ is generated by $[X_1 - F]$ and $[Y_1 - F]$.  Moreover, 
\[[X_i - F] = i[X_1 - F] \mbox{ for } i = 1,\ldots, 3m-1,\] 
and 
\[[Y_i - F] =i[Y_1 - F] \mbox{ for } i = 1, \ldots, 6m-1.\]  
We also deduce $3m[X_1 -F] = 0$ and $6m[Y_1 - F] = 0$.    Since $\Phi \simeq \Z/3m\Z \times \Z/6m\Z$, by Theorem \ref{semistable reduction}, the elements $[X_1 - F]$ and $[Y_1 - F]$ must have order $3m$ and $6m$, respectively.  In particular, the elements $[X_i - F] + [Y_j - F]$ are all non-zero. 

\begin{prop}\label{trb}
Let $p  > 3$, and assume $J = J_{a,b,c}$ has non-degenerate $b$-reduction.  Then 
\[c(A)/c(J) = \frac{1}{3} \hspace{4mm} \mbox{and} \hspace{4mm} c(B)/c(J) = 3.\]
\end{prop}
\begin{proof}
Recall $D_1 = P_1 + P_1' - K$ and $D_2 = P_2 + P_2' - K$. The homogenized polynomial $H_1(x,z)$ has two distinct roots in $\mathbb{P}^1(\FF_p)$, corresponding to the distinct nodes.  These are also roots of $G_1(x,z)$.  Hence $P_1$ and $P_1'$ specialize to different nodes on $\Cnaive/\FF_p$.  Thus, 
\[\mathrm{sp}(D_1) = [X_i - F] + [Y_i - F] \neq 0.\]  By Proposition \ref{split toric formula}, we have $c(A)/c(J) = 1/3$.  

On the other hand, a resultant computation, as in the proof of Proposition \ref{tra}, shows that the $y$-coordinates of $P_2$ and $P_2'$ are not divisible by $p$. So $P_2$ and $P_2'$ specialize to non-singular points on $\Cnaive/\bar\FF_p$.  Thus 
\[\mathrm{sp}(D_2) = [F + F - 2F]  = 0 \in \Phi.\]
By Proposition \ref{split toric formula}, we have $c(B)/c(J) = 3$.  
\end{proof}      
      
 \subsection{Remaining cases}\label{remain}
 The remaining cases of binodal reduction can be deduced from the trinodal cases, using the Richelot isogeny of Proposition \ref{P:richelot}.
 
 \begin{prop}\label{binode}
Assume $p  > 3$ and let $J = J_{a,b,c}$. 
\begin{enumerate}
\item If $J$ has $c$-reduction, then $c(A)/c(J) = \frac13$ and $c(B)/c(J) = 3$.
\item If $J$ has $q_4$-reduction, then  $c(A_{\minus3})/c(J_{\minus3}) = 3$ and $c(B_{\minus3})/c(J_{\minus3}) = \frac13$.
\item If $J$ has $t_4$-reduction, then $c(A)/c(J) = c(B)/c(J) = \frac13$. 
\end{enumerate}
\end{prop}

\begin{proof}
The Richelot isogeny $\pi \colon J_{a,b,c} \to J_{c,b,a}$ has kernel $J[1 - \sqrt{3}]$.  It follows that $\pi$ sends $J[\sqrt{3}]$ to $J_{c,b,a}[\sqrt 3]$. 
We write $\langle D'_1\rangle$ for the subgroup of $J_{c,b,a}[\sqrt{3}]$ isomorphic to $\ZZ/3\ZZ$ and $\langle D'_2\rangle$ for the one isomorphic to $\mu_3$. Since $\pi$ is defined over the base field, we have that $\pi(\langle D_1\rangle)=\langle D_1'\rangle$ and that $\pi(\langle D_2\rangle)=\langle D_2'\rangle$.
Hence $J/\langle D_1\rangle \simeq J_{c,b,a}/\langle D_1'\rangle$ and similarly for $D_2$.  By (\ref{E:discfactors}), the involution $(a \colon b \colon c) \mapsto (c \colon b \colon a)$ exchanges $a$- and $c$-reduction, $q_3$- and $q_4$-reduction, and $t_3$- and $t_4$-reduction.  Thus, the proposition follows from the corresponding statements in Section \ref{trinode}. 
\end{proof}

The main results of this section (Propositions \ref{tra}-\ref{binode}) are summarized in Table \ref{table:tamagawa ratio}.

\begin{table}
\[ \renewcommand{\arraystretch}{1.4}
   \begin{array}{|r|c|c|} \hline
      h(a,b,c) & c(A_d)/c(J_d) & c(B_d)/c(J_d)   \\\hline
      a,b,c & \frac13   & 3         \\
      q_3,q_4 & 3  & \frac13        \\
      t_3,t_4 & \frac13 & \frac13        \\
\hline
    \end{array}
\]
\caption{Tamagawa ratios at primes $p > 3$ of non-degenerate reduction.  Here, $d$ is the unique squareclass for which $J_d$ has split toric reduction.  For the other three squareclasses, the Tamagawa ratios are 1.} \label{table:tamagawa ratio}
\end{table}

\section{Quadratic twists and elements of
  order 3 in $\Sha$}\label{twisting}
In this section we prove Theorem
\ref{sha}. Fix $r \geq 1$, and for  $C = C_{a,b,c}$ over $\Q$, let $J = J_{a,b,c}$ be its Jacobian.  We will always assume that the discriminant of $C_{a,b,c}$ is non-zero, so that $C_{a,b,c}$ is a smooth genus two curve.  Recall $\phi \colon J \to A$
has kernel $\langle D_1 \rangle \simeq \Z/3\Z$, while $\psi \colon J \to
B$ has kernel $\langle D_2 \rangle \simeq \mu_3$.  Twisting gives 3-isogenies $\phi_d \colon J_d \to A_d$ and $\psi_d \colon J_d
\to B_d$, for each squarefree integer $d$.  

The idea of the proof is to perform two independent $\sqrt{3}$-descents, one via $\phi_d$ and the other via $\psi_d$.  By imposing congruence conditions on $d$, and using the results of \cite{shnidman:RM}, we can guarantee both of the following properties.
\begin{itemize}
\item the $\psi_d$-descent yields a strong upper bound on the average rank of $J_d(\Q)$;
\item  the $\phi_d$-Selmer groups are systematically large.  
\end{itemize}
We use this to show that $\dim_{\FF_3}\Sha(A_d)[3]\geq 2r$, for a positive proportion of $d$.  To run the argument, we will use $r$ primes of $a$-reduction and $r$ primes of $t_3$-reduction.  For a given $(a \colon b \colon c)$, such primes may not exist, but there are many such primes for ``most" points $(a \colon b \colon c) \in \PP^2(\Z)$, as we consider points up to increasing height bounds.

We recall the {\it global Selmer ratio} $c(f)$, defined for any isogeny
$f \colon X \to Y$ of abelian varieties over $\Q$.
  First define the {\it local Selmer ratio} at a prime $p \leq
  \infty$ to be 
\[c_p(f) := \dfrac{\#\mathrm{coker} \left(X(\Q_p) \to Y(\Q_p)\right)}{\#\ker\left(
 X(\Q_p)\to Y(\Q_p)\right)}.\]
The global Selmer ratio is then defined as the product $c(f) = \prod_{p\leq \infty} c_p(f)$.

\begin{rmk}\label{Tamagawa ratio}
If $p \nmid \deg(f)\cdot \infty$, then $c_p(f) = c_p(Y)/c_p(X)$, the ratio of local Tamagawa numbers at $p$ \cite[Prop.\ 3.1]{shnidman:RM}. In particular, $c_p(f) = 1$ for all but finitely many primes $p$. 
\end{rmk}

Returning to the isogenies $\phi$ and $\psi$, we have the following definition and key results, which use the notion of $h$-reduction given at the beginning of Section \ref{subsec:discfactors}.  
\begin{defn}
We say $C_{a,b,c}$ has {$abc$-reduction} at $p$ if it has $h$-reduction with $h \in \{a,b,c\}$. 
We say $C_{a,b,c}$ has $q$-reduction at $p$ if it has $q_3$- or $q_4$-reduction.  
We say $C_{a,b,c}$ has {\it $t$-reduction} at $p$ if it has $t_3$- or $t_4$-reduction.  
\end{defn}

\begin{thm}\label{one or the other}
Suppose $C_{a,b,c}$ has at least $r$ primes $p > 3$ of $abc$-reduction and at least $r$ primes of $t$-reduction.  Then there exists $d \in \Z$ such that $c(\phi_d) = 3^{-2r}$ and $c(\psi_d) = 1$. 
\end{thm}

Before getting into its proof, we use Theorem \ref{one or the other} to prove a useful criterion. We first state a conjecture so that we can refer to it succinctly.

\begin{conj}[Conjecture~$(A,\ell, r)$]\label{conjecture2}
Let $A$ be an abelian variety over $\Q$, $\ell$ a prime, and $r \geq 0$.  Then the set $\{d \in \Z \, \colon \dim_{\FF_\ell}\Sha(A_d)[\ell] \geq r\}$ has positive lower natural density.  
\end{conj}

\begin{thm}\label{sha6}
Let $(a \colon b \colon c) \in \mathbb{P}^2(\Z)$. 
If $C_{a,b,c}$ has at least $r$ primes of $t$-reduction and at least $r$ primes of $abc$-reduction, then $A_{a,b,c}$ satisfies Conjecture $(A,3, 2r)$.
\end{thm}

\begin{proof}
By \cite[Thm.\ 6.2]{BKLOS:selmer}, the global Selmer
ratios $c(\phi_d)$ and $c(\psi_d)$ are determined by finitely many
congruence conditions on $d$.  Thus, Theorem \ref{one or the other} implies that there is a
positive density subset $\Sigma \subset \Z$ of squarefree integers such
that $c(\phi_d) = 3^{-2r}$ and $c(\psi_d) = 1$, for
all $d$ in $\Sigma$.  By
the proof of
\cite[Thm.\ 5.4]{shnidman:RM}, at least $50\%$ of the twists $J_d$ for
$d \in \Sigma$ have rank 0.  It follows that at least $50\%$ of twists
$ A_d$ have rank 0 as well, since $ A_d$ and $J_d$ are isogenous. 

On the other hand, let
$\phi' \colon A \to J$ be such that $\phi' \circ \phi = \sqrt{3}$.
Then \cite[Thm.\ 4.2]{shnidman:RM} gives the middle inequality in: 
\begin{equation}\label{selmer ineq}
\dim_{\FF_3}\Sel_3( A_d) \geq \dim_{\FF_3}
\Sel_{\phi'_d}( A_d) \geq -\mathrm{ord}_3 \, c(\phi_d) =  2r,
\end{equation}
which is valid for all $d \in \Sigma$. For half of such $d$ the rank
 of $ A_d$ is 0.
 When combined with (\ref{selmer ineq}), this forces $\dim_{\FF_3}\Sha( A_d) \geq 2r$, and proves the theorem. (We are free to ignore the finitely many twists such that $A_d[3](\Q) \neq 0$.)    
\end{proof}

Theorem \ref{one or the other} has an analogue which uses primes of $q$-reduction instead of primes of $abc$-reduction, with the roles and $\phi$ and $\psi$ swapped. 
\begin{thm}\label{sha2}
Let $(a \colon b \colon c) \in \mathbb{P}^2(\Z)$, and suppose $C_{a,b,c}$ has at least $r$ primes of $t$-reduction and at least $r$ primes of $q$-reduction. Then $B_{a,b,c}$ satisfies Conjecture $(A,3,2r)$.
\end{thm}

\begin{proof}
The proof is similar to that of Theorem \ref{sha6}. Alternatively, it follows from Theorem \ref{sha6}, via Proposition \ref{twist} and Corollary \ref{AtwistB}.
\end{proof}

Now we head toward the proof of Theorem \ref{one or the other}. We require some preliminary lemmas.

\begin{lemma}\label{pparity}
For each finite prime $p \neq 3$, there exists $d$ such that $c_p(\phi_d) = 1 =
c_p(\psi_d)$.  
\end{lemma}

\begin{proof}
Let $\phi' \colon A \to
J$ and $\psi' \colon B \to J$ be such that $\phi' \circ \phi = \sqrt
3 = \psi' \circ \psi$.  
Then $\ker \psi \simeq \ker \phi '$ and $\ker \phi \simeq \ker \psi'$. 
It follows that we can
choose $d \in \Q_p^\times /\Q_p^{\times 2}$ such that all four of
$c_p(\phi_d)$, $c_p(\phi_d')$, $c_p(\psi_d)$, and $c_p(\psi_d')$ are positive integers.
On the other hand, using Remark \ref{Tamagawa ratio} and the multiplicativity of Selmer ratios \cite[Lem.\ 3.5]{shnidman:RM}, we have
\[1 = c_p(J)/c_p(J) =  c_p(\sqrt3) = c_p(\phi_d)c_p(\phi_d') = c_p(\psi_d)c_p(\psi_d'),\]
so that all four of them must be equal to 1. 
\end{proof}

\begin{lemma}\label{3parity}
We have $\ord_3\, c(\phi_d) \equiv \ord_3\, c(\psi_d)$
and $\ord_3 \, c_3(\phi_d)
\not\equiv \ord_3 \, c_3(\psi_d)$ modulo $2$.
\end{lemma}

\begin{proof}
The integers $\ord_3\, c(\phi_d)$ and $\ord_3 \, c(\psi_d)$ have
the same parity, since they both have the parity of
$\dim_{\FF_3}\Sel_{\sqrt 3}(J_d)$ by an application of Poitou-Tate duality \cite[Thm.\ 4.3]{shnidman:RM}.  

The second claim is local in nature, but for the proof we write both global Selmer ratios as products of
local Selmer ratios and compare the parities of each local term.  We proceed in this indirect manner because we lack clean formulas for local Selmer ratios at primes of degenerate reduction.

We choose $d'$ in the same square class as $d$ in $\Q_3^\times$.  For each prime $p \neq 3$ dividing the conductor $N_J$ of $J$, we also demand that $c_p(\phi_{d'}) = 1 = c_p(\psi_{d'})$.  We can arrange to do this because of Lemma \ref{pparity} and the Chinese remainder theorem. If $p \nmid 3N_J$, then $c_p(\phi_{d'}) = 1 = c_p(\psi_{d'})$ since $J^{(d')}$ has a twist of good reduction \cite[Lem.\ 4.6]{BKLOS:sha}.

By construction, $\ord_3 \, c_\ell(\phi_{d'})$ and $\ord_3 \,
c_\ell(\psi_{d'})$ have the same parity for all $\ell \notin \{3,
\infty\}$.  At infinity, we have $\ord_3 \, c_\infty(\phi_{d'}) \not\equiv
\ord_3 \, c_\infty(\psi_{d'})$ since $\ker \phi \simeq \Z/3\Z$ and $\ker \psi
\simeq \mu_3$, which are non-isomorphic over $\RR$.  Using the global congruence $\ord_3 \, c(\phi_{d'}) \equiv \ord_3 \, c(\psi_{d'})$ and the definition of the global Selmer ratio as the Euler product of the local Selmer ratios, we conclude that
\[\ord_3 \, c_3(\phi_d) = \ord_3  \, c_3(\phi_{d'})\not\equiv\ord_3\, c_3(\psi_{d'})  = \ord_3 \, c_3(\psi_d) \pmod 2. \qedhere\] 
\end{proof}

\begin{rmk}
An argument similar to that in the proof of Lemma \ref{3parity} shows that for $p \notin \{ 3,\infty\}$, we have $\mathrm{ord}_3 \, c_p(\phi_d) \equiv \mathrm{ord}_3 \, c_p(\psi_d) \pmod 2$.  This can be useful in situations where it is difficult to compute $c_p(\phi_d)$ and $c_p(\psi_d)$ directly.  For example, there are cases of degenerate reduction where no twist of $J$ has semistable reduction, as mentioned in the introduction.     
\end{rmk}

\begin{lemma}
There exist $d$ such that $\{c_3(\phi_d), c_3(\psi_d)\} = \{1,3\}$.
\end{lemma}
\begin{proof}
As in the proof of Lemma \ref{pparity},  we can
choose $d \in \Q_3^\times /\Q_3^{\times 2}$ such that all four of
$c_3(\phi_d)$, $c_3(\phi_d')$, $c_3(\psi_d)$, and $c_3(\psi_d')$ are integers.
On the other hand, using \cite[Lem.\ 3.7]{shnidman:RM} and the multiplicativity of Selmer ratios \cite[Lem.\ 3.5]{shnidman:RM}, we have  
\[3 = c_3(\sqrt3) = c_3(\phi_d)c_3(\phi_d') = c_3(\psi_d)c_3(\psi_d')\]
so that all four must be either 1 or 3.  By the previous Lemma, this
means that one of $c_3(\phi_d)$ and $c_3(\psi_d)$ is 1 and the other
is 3. 
\end{proof}

\begin{proof}[Proof of Theorem $\ref{one or the other}$]  
Let $p_1, \ldots, p_r$ be primes of $t$-reduction and let $q_1,\ldots, q_r$ be primes of $abc$-reduction for $C$. We are given that $q_i > 3$ and we may assume $p_i > 3$ as well since if $p$ is a prime of $t$-reduction, then $p \equiv 1\pmod3$. Then $c_{p_i}(\phi) = c_{p_i}(\psi) = 1/3$ by the results in Sections \ref{trinode} and \ref{remain}.  On the other hand, for each $q_i$, there exists a twist $d$ such that $c_{q_i}(\phi_d) = 1/3$ and $c_{q_i}(\psi_d) = 3$ by the results in Sections \ref{trinode}-\ref{remain}. Combining the lemmas of this section and the Chinese
remainder theorem, we see that there exist $d \in \Z$ such that:
\begin{enumerate}
\item $c_\ell(\phi_d) = 1 = c_\ell(\psi_d)$ for all $\ell \notin\{3, \infty, p_1, \ldots,  p_r, q_1,\ldots, q_r\}$,
\item $c_{p_i(}\phi_d) = 1/3 = c_{p_i}(\psi_d)$, for $i = 1,\ldots, r$, 
\item $c_{q_i(}\phi_d) = 1/3$ and  $c_{q_i}(\psi_d) = 3$, for $i = 1,\ldots, r$, and 
\item $\{c_3(\phi_d), c_3(\psi_d)\} =
\{1,3\}$. 
\end{enumerate}
Note that $\{c_\infty(\phi_d),
c_\infty(\psi_d)\} = \{1, 1/3\}$, for all $d$, since $\ker \phi \simeq
\Z/3\Z$ and $\ker \psi \simeq \mu_3$. Also, $\{c_\infty(\phi_d), c_\infty(\phi_{-d})\} = \{1,1/3\}$ by \cite[Lem.\ 3.3]{shnidman:RM}.  Thus, we may even choose $d$ so that 
\[c_\infty(\phi_d)
c_3(\phi_d) = 1 = c_\infty(\psi_d)c_3(\psi_d).\]   
For this choice of $d$, we indeed have $c(\phi_d) = 3^{-2r}$ and $c(\psi_d) = 1$.
\end{proof}

Finally, we use Theorem \ref{sha6} to prove Theorem \ref{sha} of the introduction.
\begin{proof}[Proof of Theorem $\ref{sha}$]
By Theorem \ref{sha6}, it suffices to show that $100\%$ of $[a : b : c] \in \mathbb{P}^2(\Z)$, ordered by the usual height $\max\{|a|,|b|,|c|\}$, have at least $r$ distinct primes $p > 3$ dividing $a$ (resp.\ $t_3$) and not dividing any other discriminant factor.  By Lemma \ref{L:disc_sing}, this condition is equivalent to $p$ dividing $a$ (resp.\ $t_3$) and not dividing 
\[bc(c-b)(c-\zeta b)(c-\zeta^2b) = bc(c-b)(c^2 +bc+b^2),\]
(resp.\ $abc(a-c)$).  Thus, Theorem \ref{sha} follows from Theorem \ref{erdos-kac} below, a general fact about prime divisors of polynomial values.   
\end{proof}

\begin{rmk}
In the same way, we deduce that $100\%$ of the surfaces $B_{a,b,c}$ satisfy Conjecture $(A,3,r)$, using Theorem \ref{sha2} instead of Theorem \ref{sha6}. 
\end{rmk}

\begin{thm}\label{erdos-kac}
Let $F,G \in \ZZ[x,y,z]$ be non-constant homogenous polynomials with no common factor over $\bar \Q$, and fix $N \geq 1$.  Then for $100\%$ of points $(a \colon b \colon c)$ on $\PP^2(\Z)$, ordered by height, there are at least $N$ primes $p$ such that $p \mid F(a,b,c)$ and $p\nmid G(a,b,c)$.   
\end{thm}
\begin{proof}
This can be deduced from Erdos-Kac type theorems for multivariable polynomials, such as \cite[Thm.\ 1.8]{ELS}. For details, see the forthcoming preprint \cite{LOLS}.
\end{proof}
Theorem \ref{sha} implies Conjecture $(A,3,2r)$ for  $r = \min(N_\mathrm{lin},N_t)$, where $N_\mathrm{lin}$ (resp.\ $N_t)$ is the number of primes $p > 3$ of $abc$-reduction (resp.\ $t$-reduction) for $C$.   Even when $C$ has no primes $p > 3$ of $abc$-reduction, one can often prove Conjecture $(A \times B,3,2)$ by similar methods, as in the following result.

\begin{thm}\label{A32}
If $C_{a,b,c}$ has a prime of $t$-reduction, then Conjecture $(A_{a,b,c} \times B_{a,b,c}, 3,2)$ holds.
\end{thm}

\begin{proof}
Suppose $p $ is a prime of $t$-reduction.  As in Theorem \ref{sha}, we can find $d$ such that
\[\{c_\infty(\phi_d)c_3(\phi_d) , c_\infty(\psi_d)c_3(\psi_d)\} = \{3,1/3\}.\]
Note that we are only claiming an equality of sets, not an equality of ordered pairs.  Since we can always choose $d$ such that $c_q(\phi_d) = 1 = c_q(\psi_d)$ for all $q \nmid 3p\infty$, and since $c_p(\phi) = c_p(\psi) = 1/3$, it follows that we can find $d$ such that the global Selmer ratios are 
\[\{c(\phi_d), c(\psi_d)\} = \{1,1/9\}.\]
As in the proof of Theorem \ref{sha6}, this implies Conjecture $(A,3,2)$ for one of $A_{a,b,c}$ or $B_{a,b,c}$.
\end{proof}

\section{Examples}\label{examples}

We first prove Theorems~\ref{ex1} and \ref{ex2}, concerning the curve $C_{1,2,-1} \colon y^2 = 4f_{1,2,-1}(x/4)$. Explicitly,
\[ C = C_{1,2,\minus1} \colon y^2 =8x^5 - 3x^4 - 2x^3 - 7x^2 + 4x + 20.\]
The discriminant factors are 
\[a = 1, \quad b =2,\quad c=-1,\quad a-c = 2, \quad q_3 = 2^4, \quad q_4 = 2^2, \quad t_3 = 2^2, \quad t_4 = 2^2\cdot 31. \]
It follows that $J = J_{1,2,\minus1}$ has good reduction for all $p \nmid 2\cdot 31$.  By Theorem \ref{semistable reduction}, $J$ has bad semistable reduction at $31$.  Over $\FF_2$,  the projective model $y^2 = x^2(x+z)^2z^2$ has three ordinary double points; hence $J$ has split toric reduction over $\Q_2$.  It follows that $J$ has root conductor $2\cdot 31$, and indeed there is a weight 2 modular form of level $62$ whose coefficients can be read off from the zeta function of $C$. A search through modular forms of low level reveals that $J$ has minimal conductor among simple abelian surfaces over $\Q$ with $\Z[\sqrt 3]$-multiplication.\footnote{Even so, the large discriminant of $C$ prevents its appearance in the current version of the LMFDB.}  

To study Mordell-Weil ranks of the quadratic twists $J_d$, we use the results of Section 5 and general facts about local Selmer ratios \cite[\S3]{shnidman:RM}. We have $c_p(\phi_d) = c_p(\psi_d) = 1$ for $p \nmid 2\cdot 3\cdot 31 \cdot \infty$. By Proposition \ref{binode}, we have $c_{31}(\phi) = c_{31}(\psi) = 1/3$, and by Corollary \ref{non-split toric formula}, we have $c_{31}(\phi_d) = c_{31}(\psi_d) = 1$ for $d \notin \Q_{31}^{\times2}$.  
For $p  =3$, we check that $J$ has ordinary reduction, and that $D_1$ reduces to a non-trivial divisor class.  Thus $\ker \phi$ extends to an \'etale group scheme over $\Z_3$, and it follows that $c_3(\phi_d) = 1$ and $c_3(\psi_d) = 3$ for all $d$.  Corollary \ref{p=2} gives $c_2(\phi) = 1/3$ and $c_2(\psi) = 3$, whereas $c_2(\phi_d) = c_2(\psi_d) = 1$ for all non-square $d \in \Q_2^\times$.  Finally, $c_\infty(\phi_d) = \frac13$ or $1$ (resp.\ $c_\infty(\psi_d)$ = $1$ or $\frac13$) depending on whether $d$ is positive or negative.

To efficiently compile the global Selmer ratios $c(\phi_d)$ and $c(\psi_d)$, define the sets 
\begin{equation}\label{E:Tmn}
T_{m,n} = \{d \in \Z \colon c(\phi_d)  = 3^m \mbox{ and } c(\psi_d) = 3^n\}.
\end{equation}
By the computations above, the only non-empty sets are
\[T_{-3,-1} \quad T_{-2,2} \quad T_{-2,0} \quad T_{-1,1} \quad T_{-1,-1}\quad T_{0,0}.\]
\begin{thm}\label{211thm}
The average $\Z[\sqrt{3}]$-rank of $J_d(\Q)$, as $|d| \to \infty$, is at most $\frac{1819}{1512} \approx 1.203$.   
\end{thm}
\begin{proof}
On each $T_{m,n}$, we invoke \cite[Thm.\ 5.2]{shnidman:RM}, using whichever 3-isogeny gives a better bound.  For example, if $|m| \leq |n|$, then we use $\phi$-descent, and the theorem says that the average rank of $J_d(\Q)$, for $d \in T_{m,n}$, is at most $|m| + 3^{-|m|}$. We then compute a weighted sum of these bounds, each one weighted by the density of the set $T_{m,n}$ as a subset of the squarefree integers, endowed with the usual natural density.  We compute these densities using the above formulas for $c_p(\phi_d)$, for all $p$ and $d$.  The density of 
\[T_{-2,-2} = \left\{d \colon  d > 0, d \in \Q_{2}^{\times 2}, \mbox{and } d \notin \Q_{31}^{\times 2}\right\}\] is $\frac12\frac{8}{63}\frac{33}{2\cdot32} = \frac{11}{336}$, so the average rank of $J_d(\Q)$ as $|d| \to \infty$ is 
\[\left(\frac12 - \frac{11}{336}\right)\cdot \left(0 + 1\right) + \frac12\cdot\left(1 + \frac13\right)  + \frac{11}{336}\cdot\left(2 + \frac19\right) = \frac{1819}{1512}.\]
\end{proof}
\begin{cor}\label{112props}
The group $J_d(\Q)$ has rank $0$ for at least $23.3\%$ of $d$, when squarefree $d$ are ordered by absolute value.  It has rank at most $2$  for at least another $41.6\%$ of $d$; the rank is equal to $2$ for at least $20.8\%$ of these $d$.    
\end{cor}
\begin{proof}
This follows from the computations of the densities of $T_{-3,-1}$, $T_{-2,0}$ and $T_{0,0}$, given in the proof of Theorem \ref{211thm} and arguments similar to \cite[Thm.\ 1.5b-c]{shnidman:RM}.  Assuming $\Sha(J_d)$ is finite, then the $41.6\%$ of $d$ with rank at most two, have rank exactly 2.  We can prove this unconditionally for half of such $d$, by  invoking the recent work of Castella-Grossi-Lee-Skinner.  They prove that if $\ker\phi_d \not\simeq \Z/3\Z,\mu_3$ over $\Q_3$ (a density $1/2$ condition in any $T_{m,n}$), and if $\dim_{\FF_3}\Sel_{\sqrt{3}}(J_d) = 1$, then $J_d(\Q)$ has rank 2 and $\Sha(J_d)$ is finite \cite[Cor.\ 5.2.2]{CGLS:IMC}.  Their theorem is written for elliptic curves, but they note that their method works equally well for elliptic modular forms.  In particular, it applies to all $J_{a,b,c}$, since these are known to be quotients of the modular Jacobians $J_0(N)$.    
\end{proof}

It follows that almost two thirds of the curves $C_d$ satisfy $\mathrm{rk} \, J_d(\Q) \leq 2$. 
This has the nice consequence that for almost two thirds of $d$, we may apply the non-abelian Chabauty-Kim method to bound $\#C_d(\Q)$.  We obtain a uniform bound, using recent results of Balakrishnan-Dogra \cite{BD:chabauty}. 

\begin{thm}
Let $|d| \to \infty$ through integers $d \equiv 1 \pmod 3$.  Then for at least $23.3\%$ of $d$ we have $\#C_d(\Q) = 2$, and for at least another $41.6\%$ of $d$ we have $\#C_d(\Q) \leq 154133$.  
\end{thm}

\begin{proof}
By Corollary \ref{112props}, we have $\mathrm{rk} \, J_d(\Q) =0$ for at least $23.3\%$ of $d$.  Since there are two rational Weierstrass points, this means $\#C_d(\Q) = 2$ for all but finitely many such $d$ (using e.g.\ Raynaud's proof of the Manin-Mumford conjecture \cite{RaynaudMM}).  
For $41.6\%$ of $d$, the rank of $J_d(\Q)$ is 0 or 2 and $J_d$ has good reduction at $3$, since $d \equiv 1\pmod 3$.  For these we may apply \cite[Thm.\ 1.1]{BD:chabauty} with $p = 3$.   There are exactly two points in $C_d(\FF_3)$, both Weierstrass, so we apply the bound given in \cite[\S5.2]{BD:chabauty}.  The reason for the large constant $154133$ (despite the relatively few primes of bad reduction) is that the Tamagawa number over $\Z_2$ is 66.  We computed this using van Bommel's code \cite{vanBommel,magma}.     
\end{proof}
\begin{rmk}
We restrict to $d \equiv 1\pmod 3$ to apply the results of \cite{BD:chabauty}. A slight extension of their methods would give uniform results for all twists.  The bound $154133$ is certainly far from optimal; the point here is to give uniform and explicit bounds for a large proportion of twists, with the bounds given in terms of the geometry of the special fibers.  The exact constants should improve as the non-abelian Chabauty methods are refined. 
\end{rmk}

One can study rank statistics in this way for quadratic twists of any $C = C_{a,b,c}$.  The analysis is easier if $C$ has semistable reduction at 2 and 3.  It is also helps when there are no primes of degenerate reduction, since this implies semistable reduction for all $p > 3$.  

Continuing with the curve $C = C_{1,2,\minus1}$, let us next consider the Tate-Shafarevich groups $\Sha(A_d)$ and $\Sha(B_d)$, where $A =  J/\langle D_1\rangle$ and $B = J/\langle D_2\rangle$, as usual.  Theorem \ref{sha} does not quite apply, since $C$ has no primes of $abc$-reduction, but $p = 2$ plays the role of a prime of $abc$-reduction.  In particular, we can exploit the `imbalance' of the global Selmer ratios in the sets $T_{-3,-1}$ and $T_{-2,0}$, to obtain:   
\begin{thm}
At least $50\%$ of $d \in T_{-2,0}$ satisfy $A_d(\Q) \simeq \Z/2\Z$ and $\dim_{\FF_3} \Sha(A_d)[3] = 2$.  At least $83.3\%$ of $d \in T_{-3,-1}$ have $\dim_{\FF_3}\Sha(A_d)[3] \geq 2$, and for at least $41.6\%$ of those $d$, we moreover have $\dim_{\FF_3} \Sha(A_d)[3] = 2$ and $A_d(\Q) \approx \Z[\sqrt{3}] \oplus \Z/2\Z$.   
\end{thm}
 The density of $T_{-3,-1}$ is $\frac{31}{1008}$, whereas the density of $T_{-2,0}$ turns out to be $\frac{31}{126}$. 
\begin{cor}Conjecture $(A,3,2)$ holds for $A = A_{1,2,\minus1}$.  More precisely, at least $14.8\%$ of twists $A_d$ satisfy $\dim_{\FF_3}\Sha(A_d)[3] \geq 2$.
\end{cor}       

For general curves $C_{a,b,c}$, we can often prove Conjecture $(X,3,2)$ for one of $X = A_{a,b,c}$ or $X = B_{a,b,c}$ using Theorem \ref{A32}.  In fact, this almost always works: 
\begin{thm}\label{thm:exp}
Conjecture $(A_{a,b,c} \times B_{a,b,c},3,2)$ holds for all but $142$ of the $219,914$ points $(a \colon b \colon c) \in \PP^2(\Q)$ of height at most $40$ corresponding to genus two curves $C_{a,b,c}$. 
\end{thm}
\begin{proof}
We simply check whether $(a \colon b \colon c)$ satisfies the hypotheses of Theorem \ref{A32}. 
We save compute time by also imposing $|c| \leq a$, since the conjecture holds for $(a\colon b\colon c$) if and only if it holds for $(c\colon b \colon a)$, by Proposition \ref{P:richelot}.  
\end{proof}

\begin{rmk}
Theorem \ref{A32} may already be enough to prove Conjecture $(A_{a,b,c} \times B_{a,b,c}, 3,2)$ for all but finitely many $(a \colon b \colon c)\in \PP^2(\Q)$.  It reduces to asking whether there are only finitely many $(a\colon b \colon c)$ without a prime of $t$-reduction.  By Lemma \ref{L:disc_sing}, this is equivalent to the following purely algebraic question: are there 
finitely many $(a\colon b \colon c) \in \PP^2(\Z)$ such that 
\[Q(a,b\zeta,c)Q(a,b\zeta^2,c)Q(c,b\zeta,a)Q(c,b\zeta^2,a)\] 
is divisible by a prime $p$ not dividing $abc(a-c)$?
Recall, $Q(a,b,c)=a^2 +a(b-c) + (b-c)^2-3ac$.
\end{rmk}

Finally, we show how to make Theorems \ref{sha} and \ref{sha2} explicit.  The curves $C_{1,b,\minus1}$ are convenient for this purpose, since one checks that $C$ has non-degenerate reduction at all $p > 3$ dividing $bq_3q_4t_3t_4$.  
\begin{example}For $b = 2^{14}$, the discriminant factors of $C_{1,b,\minus1}$ are:
\begin{align*}
q_3 &= 2 \cdot 31 \cdot 107 \cdot 40471 \\
q_4 &= 2 \cdot 5 \cdot17 \cdot 23 \cdot 83 \cdot 827 \\
t_3 &= 2^2 \cdot 19 \cdot 691 \cdot 698779 \cdot 1963219 \\
t_4 &= 2^2 \cdot 18017697245765641.
\end{align*}  
By Theorem \ref{sha2}, for each $r \leq 5$, we have $\#\Sha(B_d)[3] \geq 3^{2r}$ for a positive proportion of $d$.  

We give explicit lower bounds for these proportions.  For simplicity, we restrict to the subset $\Sigma \subset \Z$ of positive squarefree $d$ coprime to the conductor of $J$. Define sets $T_{m,n} \subset \Sigma$ as in \eqref{E:Tmn}.  If $p$ is a prime of $q$-reduction, then $(c_p(\phi_d),c_p(\psi_d))$ equals $(3,\frac13)$ for half of $d \in \Sigma$ and $(1,1)$ for the other half.  If $p$ is a prime of $t$-reduction, we instead get $(\frac13,\frac13)$ and $(1,1)$.  

Thus $T_{m,\minus n} \neq \emptyset$  if and only if $m = i-j$ and $n = i + j$, with $0 \leq j \leq 8$ and $0 \leq i \leq 5$; so $j = \frac12(n-m)$ and $i =\frac12(n+m)$.  
The density of $T_{m,\minus n} \subset\Sigma$ is therefore $2^{-13} {8\choose j} {5\choose i}$. 
Quantifying the proof of Theorem \ref{sha2}, we find that at least $1 - \frac{1}{2\cdot 3^{|m|}}$ of $d \in T_{m,\minus n}$ satisfy $\Sha(B_d)[3]~\geq~3^{n - |m|}$. From this we can readily compute the lower bounds on the proportion of $d \in \Sigma$ with $\#\Sha(B_d)[3] \geq 3^{r}$ that appear in Table \ref{shatable} below. 
\end{example} 

\begin{table}[ht]
\caption{Proportion of $d \in \Sigma$ such that $\#\Sha(B_d)[3] \geq 3^r$ for $B = B_{1,2^{14},-1}$.}
\centering
\begin{tabular}{cc}
\hline \hline
$r$ & Proportion (lower bound)\\
\hline
2 & .820 \\
4 & .649\\
6 & .335\\
8 & .085\\
10 & .007\\
\hline
\end{tabular}
\label{shatable}
\end{table}


\end{document}